\documentclass[11pt,reqno]{amsart}
\pdfoutput=1

\usepackage{mathtools,cases}
\usepackage{amssymb,amscd,amsfonts,amsbsy,enumerate}
\usepackage{latexsym}
\usepackage{exscale}
\usepackage{amsmath,amsthm,amsfonts}
\usepackage{mathrsfs}
\usepackage{graphics}
\usepackage{esint}
\usepackage{enumitem} 
\usepackage{tikz}
\usepackage{pgfplots}
\pgfplotsset{compat=1.10}
\usepgfplotslibrary{fillbetween}
\usetikzlibrary{patterns}
\usepackage[colorlinks=true, pdfstartview=FitV, linkcolor=blue, citecolor=red, urlcolor=blue]{hyperref}

\usepackage{color}

\numberwithin{equation}{section}

\allowdisplaybreaks

\newtheorem{theorem}{Theorem}[section]
\newtheorem{lemma}[theorem]{Lemma}

\newtheorem{proposition}[theorem]{Proposition}
\newtheorem{definition}[theorem]{Definition}

\newtheorem*{step1}{Step 1}
\newtheorem*{step2}{Step 2}
\newtheorem*{step3}{Step 3}
\theoremstyle{remark}
\newtheorem{remark}[theorem]{Remark}

\newcommand{\rn}{\mathbb{R}^{N}}

\newcommand{\X}{\mathbf{X}}

\newcommand{\Y}{\mathscr{L}^{-1}_{2,N-2}}
\newcommand{\Z}{\mathbf{Z}}

\newcommand{\Ball}{|B(x,R)|}

\newcommand{\IntC}{\int_{0}^{R^2}\int_{B(x,R)}}

\newcommand{\dv}{\mbox{div}\hspace{0.1cm}}

\begin{document}
	\allowdisplaybreaks
	
	\title[]
	{Well-posedness for chemotaxis-fluid models in  arbitrary dimensions}

	\author{Gael Yomgne Diebou}
	\address{ The Fields Institute for Research in Mathematical Sciences
		\\
		222 College Street, 2nd floor, Toronto, Ontario
		\\
		M5T 3J1 Canada} \email{gyomgned@fields.utoronto.ca}

	\thanks{The author acknowledges the support of 
		the DAAD through the program ''Graduate School Scholarship Programme, 2018'' (Number 57395813) and the  Hausdorff Center for Mathematics at Bonn}

	\date{\today}
	
	
	
	\begin{abstract}
		We study the Cauchy problem for the chemotaxis Navier-Stokes equations and the Keller-Segel-Navier-Stokes  system. Local-in-time and global-in-time solutions satisfying fundamental properties such as mass conservation and nonnegativity preservation are constructed for low regularity data in $2$ and higher dimensions under suitable conditions. Our initial data classes involve a new scale of function space, that is $\Y(\rn)$ which collects  divergence of vector-fields with components in the square Campanato space $\mathscr{L}_{2,N-2}(\rn)$, $N>2$ (and can be identified with the homogeneous Besov space $\dot{B}^{-1}_{22}(\rn)$ when $N=2$) and are shown to be optimal in a certain sense. Moreover, uniqueness criterion for global solutions is obtained under certain limiting conditions.
	\end{abstract}
	
	\maketitle
	
	\allowdisplaybreaks

	\normalsize
	
	\section{Introduction}
	\setcounter{equation}{0}
	\label{sec:1}
	Micro-organisms (e.g. bacteria) have very limited ability to adapt to fluid environments due to their small size. They respond to detectable change by swimming towards specific regions. The orientation mechanism by which they approach  or are repelled from a chemical source is known as chemotaxis. When the fluid is incompressible, upon assuming that swimmers contribute at a very small scale to the swimmers-fluid suspension and that hydrodynamics interactions between swimmers (e.g. cell-cell interaction, which can lead to  collective motion, see for instance \cite{CM} and cited works therein) are negligible, the authors in  \cite{Tu} proposed the following mathematical model
	\begin{align}\label{ME} 
		\begin{cases}
			\partial_tc-D_c\Delta c+ nf(c) +u\cdot\nabla c=0&\mbox{in}\hspace{0.2cm}\Omega\times (0,T)\\
			\partial_tn-D_n\Delta n+\rho\dv (n\chi(c)\nabla c)+u\cdot\nabla n=0&\mbox{in}\hspace{0.2cm}\Omega\times (0,T)\\
			\partial_tu+u\cdot \nabla u-\nu\Delta u+n\nabla\Phi+ \nabla p=0&\mbox{in}\hspace{0.2cm}\Omega\times (0,T)\\
			\dv u=0&\mbox{in}\hspace{0.2cm}\Omega\times (0,T)
		\end{cases}   
	\end{align}
	where $c$ is the concentration of oxygen, $n$ is the density of cells, $u$ is the velocity field of the fluid governed by the incompressible Navier-Stokes equations with scalar pressure $p$ and viscosity $\nu$. The time-independent  gravitational force exerted from a bacteria onto the fluid is modelled through $\nabla\Phi$. The constant $\rho$ represents the magnitude of chemotaxis and $D_c,D_n$ are diffusion coefficients. 
	The function $f(c)$ models the inactivity level caused by a low supply of oxygen  and $\chi(c)$ is a suitable cut-off function (usually determined by experiments). The second equation in \eqref{ME} describes the mass balance equation for the cells combining the advection effect (modelled through $u\cdot \nabla n$), the chemotactic effect or the migration towards regions of high concentration of oxygen (modelled by $\dv (n\chi(c)\nabla c)$) and the diffusion of cells (modelled through $D_n\Delta$).

	The analysis of the Cauchy problem for Syst. \ref{ME} seems challenging from a mathematical point of view and a lot of effort over the past recent years have been devoted to the understanding of its dynamics with a particular focus on the existence of local and global solutions as well as their qualitative behaviour (long-time asymptotic, stability, blow-up,...). Some of the main challenges arising in the analysis of Syst. \ref{ME} are inherited from the Navier-Stokes equations. When $\Omega\subset \rn$ is a (sufficiently) smooth bounded domain, upon neglecting the contribution of the convection term $u\cdot \nabla u$ in \eqref{ME}, Lorz \cite{Lo} obtained the existence of local-in-time weak solutions for the associated initial boundary value problem  with no-flux  boundary conditions in dimensions $N=2,3$ for $\chi(c)\equiv const.$ under some monotonicity and differentiabilily condition on $f$. Still in absence of the convection term, global well-posedness in $\mathbb{R}^2$ was proved by Duan, Lorz \& Markowich \cite{DLM} for non-constant smooth $\chi$ under a smallness assumption on either the gravitational force $\nabla \Phi$ or the initial concentration $c_0$. Moreover, in presence of a convection term, they established the existence of classical solutions in $\mathbb{R}^3$ using uniform a priori estimates  under a suitable smallness condition on the initial data in $H^3(\mathbb{R}^3)$ and derived time-decay rate of solutions near constant steady states. For $\Omega\subset \mathbb{R}^3$ smooth and bounded, Winkler \cite{Wi3} constructed global weak solutions  under some structural and strong smoothness assumptions on $f$ and $\chi$. Under very similar requirements, the same author in \cite{Wi2} introduced the notion of eventual energy solutions and proved that the initial-boundary value problem for Syst. \ref{ME} (with homogeneous Neumann conditions) admits at least one such solution. Existence of smooth local solutions in higher order Sobolev space and blow-up issues have been considered by Chae, Kang \& Lee \cite{CKL} for $N=2,3$. Their result was later extended by Zhang in \cite{Zha} in the framework of Besov spaces by means of Fourier localisation technique. We point out, however that although global weak solution exist in three dimensions, the question of their uniqueness seems unsolved. 
	Regarding large data global existence, we quote the works \cite{LL,Wi1,ZZ} and references therein. Other interesting related models with inhomogeneous tensor-valued chemotactic sensitivity can be found in \cite{CL,Wi4}.  A popular model considers the choices $f(c)=c$, $\chi(c)=1$; $D_c=D_n=\rho=\nu=1$ and $\Omega=\rn$ turning \eqref{ME} into the system
	\begin{align}\label{main-eq} \tag{CNS}
		\begin{cases}
			\partial_tc-\Delta c+c n+u\cdot\nabla c=0&\mbox{in}\hspace{0.2cm}\rn\times (0,T)\\
			\partial_tn-\Delta n+\dv (n\nabla c)+u\cdot\nabla n=0&\mbox{in}\hspace{0.2cm}\rn\times (0,T)\\
			\partial_tu-\Delta u+u\cdot \nabla u+n\nabla\Phi+ \nabla p=0&\mbox{in}\hspace{0.2cm}\rn\times (0,T)\\
			\dv u=0&\mbox{in}\hspace{0.2cm}\rn\times (0,T).
		\end{cases}   
	\end{align}
	where $T\in (0,\infty]$.
	Recently, small data global existence and large data local existence in critical Besov space have been investigated in \cite{CLY}. Kozono, Miura \& Sugiyama \cite{KMS} obtained global existence and large time asymptotic behaviour of mild solutions to \eqref{main-eq} for initial data \[[c_0,n_0,u_0]\in L^{\infty}(\rn)\times L^{N/2,\infty}(\rn)\times L^{N,\infty}(\rn)\hspace{0.2cm} \mbox{with}\hspace{0.2cm} \nabla c\in L^{N,\infty}(\rn)\] when $N\geq 3$ (and $n_0\in L^{1}(\mathbb{R}^2)$ when $N=2$) having a sufficiently small norm. Their proof relies on heat semigroup estimates in weak-Lebesgue spaces combined with the implicit function theorem. Using an argument based on Picard iteration, the authors in \cite{YFS} (see also \cite{FP} for similar results pertaining to a generalized model) extended the latter result by considering small data 
	$[c_0,n_0,u_0]\in L^{\infty}(\rn)\times \mathcal{N}^{-s_1}_{p_1,\lambda,\infty}(\rn)\times \mathcal{N}^{-s_2}_{p_2,\lambda,\infty}(\rn)\hspace{0.2cm} \mbox{with}\hspace{0.2cm} \nabla c\in \mathcal{N}^{-s_3}_{p_3,\lambda,\infty}(\rn)$ for some parameters $s_j>0$ depending on $0\leq \lambda<N$, $N\geq 2$ and $1<p_j<\infty$, $j=1,2,3$.  Here $\mathcal{N}^{-s}_{p,\lambda,\infty}(\rn)$ stands for the homogeneous Besov-Morrey space (see below for the definition). We note that in the aforementioned results, the extra assumption on the gradient of the initial concentration plays a central role and as we will show later, this hypothesis is unnecessary and can be discarded.  A common feature among these works is that they are all obtained in critical spaces, i.e. invariant under natural scaling in contrast to energy solutions. It is therefore interesting to find the largest critical space where solutions exist and satisfy expected properties including mass conservation and nonnegativity in the first two components of the solution.
	
	The main purpose of this paper is the study of the well-posedness for the Cauchy problem \eqref{main-eq} and its generalized model (cf. Syst. \ref{Sd-eq} in Section \ref{sec:2}). In particular, we are interested in the existence of small data global-in-time and large data local-in-time solutions in the largest scaling and translation invariant function spaces.  The method carried out here is dimension independent so that our results are valid in any space dimension larger or equal to two. To motivate our study, we briefly comment on the work \cite{KT} by Koch \& Tataru. Introducing a new function space (see $\X_3$ below) based on the intrinsic properties of solutions, they investigated the local and global well-posedness issues for the incompressible Navier-Stokes equations (the third equation in \eqref{main-eq} with $\Phi=0$). More precisely, they proved existence of a unique small global solution under the conditions that the initial data is divergence-free and has small $BMO^{-1}$-norm and of small local solutions for divergence-free initial data in $\overline{VMO^{-1}}$. As observed by the authors in \cite{CG}, their results seem to be the endpoint case for small data global existence. 
	Coming back to \eqref{main-eq}, we carefully analyse each coupling term to find necessary conditions on each component for which the equation is meaningful. This leads to well-posedness results which are optimal in a sense which is made precise later. In addition, it is proved that mass (of the initial density $n_0$) and nonnegativity of $c_0$ and $n_0$ are conserved in finite time. Proving that mild solutions to \eqref{main-eq} and \eqref{Sd-eq} satisfy these basic properties is nontrivial and have not been addressed in any of the cited references above. We argue here that the framework within which the analysis is carried out plays a crucial role in establishing such properties. In fact, solutions we construct can be shown to be equivalent to some notion of weak solutions (provided, of course that extra assumptions are put on $n_0$ and $u_0$) since each component of the solution, the gradient of the first component, belongs to $L^2_{loc}$ in space and time in addition to the force being in $L^2_{loc}$. We believe this observation might be useful in many respects (e.g. nonuniqueness of weak solutions under no smallness assumption on the force). We discuss at the end of Section \ref{sec:2} possible applications of our results in connection with regularity and stability of solutions to chemotaxis fluid models.       
	\section{Function spaces and main results}\label{sec:2}
	System \eqref{main-eq} is scaling and translation invariant provided $\Phi\in \mathcal{S}'(\rn)$ is such that $\nabla\Phi$ is homogeneous of degree $-1$. More precisely, if $[c,n,u]$ solves \eqref{main-eq} (in a classical sense), then $[c_{\delta},n_{\delta},u_{\delta}]$ with
	\begin{equation}\label{self-simi}
		c_{\delta}(x,t)=c(\delta x,\delta^2t), \hspace{0.12cm}n_{\delta}(x,t)=\delta^2n(\delta x,\delta^2t),\hspace{0.12cm}u_{\delta}(x,t)=\delta u(\delta x,\delta^2t),\quad \delta>0   
	\end{equation}
	is another solution. On the other hand, a weaker requirement on the unknowns $c,n$ and $u$ for \eqref{main-eq} to make sense is that 
	\begin{align}\label{min-cond}
		\begin{cases}
			c\in L^{\infty}_{loc}(\rn\times [0,\infty)), \nabla c\in L^2_{loc}(\rn\times (0,\infty))\\ 
			n\in L^2_{loc}(\rn\times (0,\infty))\\
			u\in L^2_{loc}(\rn\times (0,\infty)).
		\end{cases}    
	\end{align} 
	Thus we look for initial data $[c_0,n_0,u_0]$ whose caloric extension $[\widetilde{c_0},\widetilde{n_0},\widetilde{u_0}]$ satisfy the scaling and translation invariant analogue of condition \eqref{min-cond}, that is,
	\begin{subequations}\label{min-cond-1}
		\begin{numcases}{}
			\label{cd:c}
			\sup_{t>0}\|\widetilde{c_0}(t)\|_{L^{\infty}(\rn)}+\sup_{x,R>0}R^{-N}\int_{B_R(x)}\int_0^{R^2}|\nabla \widetilde{c_0}(y,t)|^2dtdy<\infty\\ 
			\label{cd:n}
			\sup_{x,R>0}R^{2-N}\int_{B_R(x)}\int_0^{R^2}|\widetilde{n_0}(y,t)|^2dtdy<\infty\\
			\label{cd:u}
			\sup_{x,R>0}R^{-N}\int_{B_R(x)}\int_0^{R^2}|\widetilde{u_0}(y,t)|^2dtdy<\infty.
		\end{numcases}
	\end{subequations}
	It is well-known that the finiteness of the second term in \eqref{cd:c}  is equivalent to  $c_0$ being an element of $BMO(\rn)$ with the equivalence of semi-norms, see for instance \cite{St}. However, the requirement that $c$ be  bounded in space and time rules out the choice of $c_0$ in $BMO(\rn)$. It rather seems plausible to prescribe the initial data $c_0$ in a subclass namely, in $L^{\infty}(\rn)$. On the other hand, condition \eqref{cd:u} is equivalent to $u_0\in BMO^{-1}(\rn)$, see \cite{KT}.  By analogy to the latter  cases, one would like to relate condition \eqref{cd:n} to some class of  functions defined on $\rn$ in an extrinsic manner.
	\begin{definition}\label{def:M-1}
		Let $N>2$. A tempered distribution $f$ on $\rn$ is an element of $\Y(\rn)$ if
		its caloric extension $\widetilde{f}=e^{t\Delta} f$ satisfies
		\begin{equation}\label{norm-Y}
			\|f\|_{\Y(\rn)}:=\sup_{x,R>0}\bigg(\Ball^{2/N-1}\IntC|\widetilde{f}(y,t)|^2dydt\bigg)^{1/2}<\infty.    
		\end{equation}
	\end{definition}
	Carleson measures characterization of square Campanato spaces (see \cite{JXY}) suggests that $\Y(\rn)$ may be regarded as the space of derivatives of distributions in the Campanato class $\mathscr{L}_{2,N-2}(\rn)$.  
	\begin{lemma}
		A tempered distribution $f$ belongs to $\Y(\rn)$ if and only if there exists  $f_j\in \mathscr{L}_{2,N-2}(\rn)$, $j=1,...,N$ such that $f=\displaystyle\sum_{j=1}^N \partial_jf_j$. 
	\end{lemma}
	The proof of this lemma is postponed to the Appendix for convenience. Let $r>0$ and  the open ball $B_r(x)=\{y\in \rn: |y-x|<r\}$. For $1\leq p< \infty$, $0\leq \mu<N$, recall the Morrey space $M_{p,\mu}(\rn)$ defined as 
	\begin{align}\label{Morrey}
		M_{p,\mu}(\rn)=\{f\in L^p_{loc}( \rn): \|f\|_{M_{p,\mu}(\rn)}<\infty\}    
	\end{align}
	where 
	\begin{align}\label{l-norm_morrey}
		\|f\|_{M_{p,\mu}(\rn)}:=\sup_{x\in\rn,r>0}r^{-\frac{\mu}{p}}\|f\|_{L^{p}(B_r(x))}.    
	\end{align}
	With $\overline{f}_{B_r(x)}:=\fint_{B_r(x)}f(y)dy$ denoting the integral mean of $f$ over the ball $B_r(x)$, the Campanato space $\mathscr{L}_{p,\lambda}(\rn)$, $\lambda\in [0,N+p)$ collects all locally integrable functions $f$ such that $\|f\|_{\mathscr{L}_{p,\lambda}(\rn)}$  is finite where 
	\begin{align}\label{campanato-norm}
		\|f\|_{\mathscr{L}_{p,\lambda}(\rn)}=\sup_{x\in \rn, r>0}\bigg(r^{-\lambda}\int_{B_r(x)}\big|f(x)-\overline{f}_{B_r(x)}\big|^p\bigg)^{1/p}.    
	\end{align}
	The expression in \eqref{campanato-norm} defines a semi-norm on $\mathscr{L}_{p,\lambda}(\rn)$ and upon identifying functions which differ by a real constant, this space becomes Banach. Note that $\mathscr{L}_{p,N}(\rn)\simeq BMO(\rn)$
	where $BMO(\rn)$ is the  space of bounded mean oscillations and whenever $N<\lambda<N+p$, the space $\mathscr{L}_{p,\lambda}(\rn)$ is equivalent to the homogeneous H\"{o}lder space $C^{0,\alpha}(\rn)$, $\alpha=\frac{\lambda-N}{p}\in (0,1)$ modulo constants. 
	We also define the local Campanato space $\mathscr{L}_{p,\lambda;R}(\rn)$ (resp. local Morrey space $M_{p,\lambda;R}(\rn)$) by taking in \eqref{campanato-norm} (resp. in \eqref{l-norm_morrey}) balls of radius $R$ and smaller. The space $\mathscr{L}^{-1}_{2,N-2;R}(\rn)$ is defined analogously and we use the notation $\mathscr{L}^{-1}_{2,N-2}(\rn)$ for $\mathscr{L}^{-1}_{2,N-2;\infty}(\rn)$.
	\begin{definition}
		A tempered distribution $f$ is an element of $BMO^{-1}(\rn)$  if there exists $f_j\in BMO(\rn)$, $j=1,...,N$ such that $f=\sum_{j=1}^{N}\partial_jf_j$. This space is equipped  with the norm
		\begin{equation*}
			\|f\|_{BMO^{-1}(\rn)}=\inf\bigg\{\sum_{j=1}^N\|f_j\|_{BMO(\rn)}: f=\sum_{j=1}^N\partial_jf_j\bigg\}.    \end{equation*}
		The local space $BMO^{-1}_R(\rn)$ is defined similarly as above by replacing the $BMO$ semi-norm by its local version. The Sarason space of vanishing mean oscillations is defined as 
		\begin{align*}
			VMO(\rn)=\big\{h\in BMO(\rn):\lim_{R\rightarrow 0}\|h\|_{BMO_R(\rn)}=0\big\}.    
		\end{align*}
		We say that $f\in \overline{VMO^{-1}}(\rn)$ if 
		\begin{align*}
			\lim_{R\rightarrow 0}\|f\|_{BMO^{-1}_{R}(\rn)}=0.    
		\end{align*}
		Let $f\in \mathscr{L}^{-1}_{2,\lambda;1}(\rn)$. We say that $f$ belongs to  $\overline{V\mathscr{L}_{2,\lambda}^{-1}}(\rn)$ if 
		\begin{align*}
			\lim_{R\rightarrow 0}\|f\|_{\mathscr{L}^{-1}_{2,\lambda;R}(\rn)}=0.    
		\end{align*}
	\end{definition}

	Recall the definition of Besov-Morrey spaces \cite{KY,M}. Let $\psi$ be a Schwartz function  supported in the annulus $1/2\leq |\xi|\leq 2$ such that $$\displaystyle
	\sum_{j\in \mathbb{Z}}\psi(2^{-j}\xi)=1,\quad \xi\in \rn\setminus\{0\}.$$ 
	Let $\mathcal{F}$ and $\mathcal{S}_{0}'(\rn)$ denote respectively, the Fourier transform and the space of Schwartz distributions in $\rn$ modulo polynomials.
	Denote by $\dot{\Delta}_j$ the Littlewood-Paley projection operator defined as $\dot{\Delta}_jf=\mathcal{F}^{-1}(\psi(2^{-j}\cdot)\mathcal{F}f)$.   The homogeneous Besov-Morrey space $\mathcal{N}^{s}_{p,\lambda,q}(\rn)$ for $s\in \mathbb{R}$, $0\leq \lambda<N$ and $p,q\in [1,\infty]$ is defined as  
	\begin{align*}
		\mathcal{N}^{s}_{p,\lambda,q}(\rn):=\{f\in \mathcal{S}_{0}'(\rn):\|f\|_{\mathcal{N}^{s}_{p,\lambda,q}(\rn)}<\infty\}    
	\end{align*}
	where 
	\[\|f\|_{\mathcal{N}^{s}_{p,\lambda,q}(\rn)}=\begin{cases}
		\bigg(\displaystyle\sum_{j\in \mathbb{Z}}\big(2^{js}\big\|\dot{\Delta}_jf\big\|_{M_{p,\lambda}(\rn)}\big)^{q}\bigg)^{1/q}<\infty,\quad q\in [1,\infty)\\
		\displaystyle\sup_{j\in \mathbb{Z}} \big(2^{js}\big\|\dot{\Delta}_jf\big\|_{M_{p,\lambda}(\rn)}\big),\quad q=\infty.
	\end{cases}\]   
	These spaces were introduced by Kozono and Yamazaki in \cite{KY} and can be regarded as straightforward extensions of homogeneous Besov spaces. As a matter of fact, one has $$\mathcal{N}^{s}_{p,0,q}(\rn)=\dot{B}^{s}_{pq}(\rn),\quad1\leq p,q\leq \infty.$$
	
	\begin{definition}Let $T\in (0,\infty]$. We say that a function $v:\rn\times \mathbb{R}_+\rightarrow \mathbb{R}$, $N>2$ belongs to $\X_{j,T}$, $j=1,2$ if $\|v\|_{\X_{j,T}}$ is finite, 
		\begin{equation*}
			\|v\|_{\X_{1,T}}=\sup_{0<t\leq T}\|v(t)\|_{L^{\infty}(\rn)}+[v]_{\X_{1,T}}
		\end{equation*}
		where
		\begin{equation*}
			[v]_{\X_{1,T}}=\sup_{0<t\leq T}t^{\frac{1}{2}}\|\nabla v(t)\|_{L^{\infty}(\rn)}+\sup_{x\in \rn,0<R\leq T^{\frac{1}{2}}}\bigg(\Ball^{-1}\IntC|\nabla v(y,t)|^2dydt\bigg)^{1/2},   
		\end{equation*}
		and 
		\begin{equation*}
			\|v\|_{\X_{2,T}}=\sup_{0<t\leq T}t\|v(t)\|_{L^{\infty}(\rn)}+\sup_{x\in \rn,0<R\leq T^{\frac{1}{2}}}\bigg(\Ball^{\frac{2}{N}-1}\IntC|v(y,t)|^2dydt\bigg)^{1/2}.    
		\end{equation*}
		The space $\X_{3,T}$ collects all functions $u:\rn\times \mathbb{R}_+\rightarrow \rn$ such that 
		\begin{equation*}
			\|u\|_{\X_{3,T}}=\sup_{0<t\leq T}t^{\frac{1}{2}}\|u(t)\|_{L^{\infty}(\rn)}+\sup_{x\in \rn,0<R\leq T^{\frac{1}{2}}}\bigg(\Ball^{-1}\IntC|u(y,t)|^2dydt\bigg)^{\frac{1}{2}}
		\end{equation*}
		is finite. One simply writes $\X_{j}$ instead of $\X_{j,\infty}$ and adopt the notation $[\cdot]_{\X_{j,\infty}}=[\cdot]_{\X_{j}}$.
	\end{definition}
	It can be easily verified that each of the spaces $\X_{j,T}$ is a Banach space when endowed with the norm $\|\cdot\|_{\X_{j,T}}$, $j=1,2,3$ respectively.  We also remark that $\X_1$ is the parabolic version of the framework used in \cite{YK} for the analysis of the weakly harmonic maps problem and $\X_3$, the Koch-Tataru space \cite{KT}. The above discussion motivates the choice of the class $\X_0$, comprising 3-tuples  $[c_0,n_0,u_0]$ such that
	\begin{equation}\label{ini-cd}
		c_0\in L^{\infty}(\rn),\hspace{0.2cm}n_0\in \Y(\rn)\hspace{0.2cm}\mbox{and}\hspace{0.2cm}u_0\in BMO^{-1}(\rn,\rn).  
	\end{equation}
	Now for $T\in (0,\infty]$, define the spaces \[\X_T=\X_{1,T}\times \X_{2,T}\times \X_{3,T} \hspace{0.2cm}\mbox{and} \hspace{0.2cm}\X_0=L^{\infty}(\rn)\times \Y(\rn)\times BMO^{-1}(\rn)\] with their respective norms  
	\begin{align*}\label{normsX}
		\big\|[c,n,u]\big\|_{\X_T}&:=\|c\|_{\X_{1,T}}+\|n\|_{\X_{2,T}}+\|u\|_{\X_{3,T}},\\    
		\big\|[c_0,n_0,u_0]\big\|_{\X_0}&:=\|c_0\|_{L^{\infty}(\rn)}+\|n_0\|_{\Y(\rn)}+\|u_0\|_{BMO^{-1}(\rn)}. 
	\end{align*}
	When $N=2$, the space $\Y(\rn)$ is replaced by the homogeneous Besov space $\dot{B}^{-1}_{2,2}(\rn)$ and we instead define $\X_{2}$ with the norm 
	\begin{equation*}
		\|v\|_{\X_{2}}=\sup_{t>0}t\|v(t)\|_{L^{\infty}(\mathbb{R}^2)}+\|v\|_{L^2(\rn\times (0,\infty))}.  
	\end{equation*}
	In what follows, $UC(\rn)$ stands for the space of uniformly continuous (real-valued) functions in $\rn$ and $\overline{UC(\rn)}^{L^{\infty}(\rn)}$ is the closure of $UC(\rn)$ in the $L^{\infty}$-norm.
	
	The main results of this paper read as follows.
	\begin{theorem}[Local well-posedness]\label{thm:LWP}
		Let $\Phi$ be a Schwartz distribution such that $\nabla \Phi\in M_{2,N-2}(\rn)$, $N\geq 2$ and $d_0\in UC(\rn)$. For any $u_0$ divergence-free distribution in $\overline{VMO^{-1}}(\rn)$ and for all $[c_0,n_0]$ in $\overline{UC(\rn)}^{L^{\infty}(\rn)}\times \overline{V\mathscr{L}_{2,N-2}^{-1}}(\rn)$, there exist $\delta_0:=\delta_0(d_0)>0$, $T:=T(\delta_0)>0$ and $[c,n,u]$ in  the space $(\varGamma_{\delta_{0}}+\X_{1,T^2})\times \X_{2,T^2}\times \X_{3,T^2}$ solving Eqs. \ref{main-eq} in $(0,T)\times \rn$ provided $\nabla \Phi$ is small enough in the norm of $M_{2,N-2}(\rn)$. Here $\varGamma_{s}=e^{s^2\Delta}d_0$, $s>0$.
	\end{theorem}
The solution constructed in the previous theorem preserves mass and the nonnegativity. In  more details, the statement reads as follows. 
	
	\begin{proposition}\label{prop:mass-con}
		Consider the hypotheses of Theorem \ref{thm:LWP} and $n_0\in L^1(\rn)\cap \overline{V\mathscr{L}_{2,N-2}^{-1}}(\rn)$. Then there exist $T>0$ and a unique solution $(c,n,u)$ to Eqs. \ref{main-eq}  where the component $n\in C([0,T);L^1(\rn))\cap \X_{2,T^2}$ and obeys the property
		\begin{equation}\label{mass conservation}
			\int_{\rn}n(x,t)dx=\int_{\rn}n_0(x)dx,\quad t\in (0,T).    
		\end{equation}
		If $c_0,n_0\geq 0$ then $c(x,t)\geq 0$ and $n(x,t)\geq 0$ almost everywhere for $(x,t)\in \rn\times [0,T)$. 
	\end{proposition}
	Observe that the nonnegativity preservation for $n$ may directly be deduced from \eqref{mass conservation} and an $L^1$-contraction property pertaining to solutions of \eqref{main-eq}$_2$, see Section \ref{sec:4}.  
	\begin{theorem}[Global well-posedness]\label{thm:GWP}
		Let $N\geq 2$ and assume that $\Phi$ is as in Theorem \ref{thm:LWP}. There exists $\varepsilon>0$ such that for every 3-tuple  $[c_0,n_0,u_0]$ in $\X_0$ with $\nabla\cdot u_0=0$, if it holds that 
		\begin{equation}
			\big\|[c_0,n_0,u_0]\big\|_{\X_0}+\|\nabla\Phi\|_{M_{2,N-2}(\rn)}<\varepsilon,
		\end{equation} 
		then there exists a global solution $[c,n,u]\in \X$ of \eqref{main-eq} continuously depending on the initial data. This solution is unique in the closed ball \[B^{\X}_{C\varepsilon}:=\big\{[c,n,u]\in\X: \big\|[c,n,u]\big\|_{\X}\leq C\varepsilon\big\}\] for some constant $C>0$. Moreover, if $n_0\in L^1(\mathbb{R}^2)$ with sufficiently small norm, then there exists a global solution $(c,n,u)$ to Eqs. \ref{main-eq} with $n\in C([0,\infty);L^1(\mathbb{R}^2))\cap \X_2$ such that 
		\begin{equation}\label{mc-global}
			\int_{\mathbb{R}^2}n(x,t)dx=\int_{\mathbb{R}^2}n_0(x)dx.    
		\end{equation}
	\end{theorem}
	A fine property of (local and global) solutions we constructed is the non-explosion of their $L^{\infty}$-norm  in finite time in the sense that for any $0 <T_0 < T$, we have
	\begin{equation*}
		\sup_{0\leq t<T}\|c(t)\|_{L^{\infty}(\rn)}+\sup_{T_0<t<T}(\|n(t)\|_{L^{\infty}(\rn)}+\|u(t)\|_{L^{\infty}(\rn)})<\infty.    
	\end{equation*}
	In fact the corresponding time decay rate of solutions is sharp since they match with that of the solution to the linear heat equation.
	The restriction $N=2$ in the second part of Theorem \ref{thm:GWP} naturally comes from the scaling invariance of $n$. 
	
	Our next result deals with the uniqueness of mild solutions constructed in Theorem \ref{thm:GWP}.
	\begin{theorem}\label{thm:uniqueness}
		Let $\Phi\in \mathcal{S}'(\rn)$ and $U_0=[c_0,n_0,u_0]\in \X_0$ such that $\nabla\cdot u_0=0$. Assume that $[c_1,n_1,u_1]$ and $[c_2,n_2,u_2]$ are two global mild solutions of \eqref{main-eq} in $L^{\infty}_{loc}((0,\infty);L^{\infty}(\rn))$ with initial data $U_0$. If it holds that 
		\begin{align}\label{limit-cond}
			\lim_{T\rightarrow 0}\big\|[c_1,n_1,u_1]\big\|_{\X_{T}}=0,\quad \lim_{T\rightarrow 0}\big\|[c_2,n_2,u_2]\big\|_{\X_{T}}=0,    
		\end{align}
		then $[c_1,n_1,u_1]=[c_2,n_2,u_2]$ on $\rn\times [0,\infty)$ provided $\nabla \Phi \in M_{2,N-2}(\rn)$ with sufficiently small norm.
	\end{theorem}
	Moving on, we study a more general model known as the double chemotaxis system    
	\begin{align}\label{Sd-eq}\tag{D-CNS}  
		\begin{cases}
			\partial_tc-\Delta c+c n+u\cdot\nabla c=0&\mbox{in}\hspace{0.2cm}\rn\times \mathbb{R}_+\\
			\partial_tn-\Delta n+u\cdot \nabla n+\dv (n\nabla c)+\dv (n\nabla v)=0&\mbox{in}\hspace{0.2cm}\rn\times \mathbb{R}_+\\
			\partial_tv-\Delta v+u\cdot \nabla v+\kappa v-n=0&\mbox{in}\hspace{0.2cm}\rn\times \mathbb{R}_+\\
			\partial_tu-\Delta u+u\cdot \nabla u+\nabla p+n\Psi=0&\mbox{in}\hspace{0.2cm}\rn\times \mathbb{R}_+\\
			\dv u=0&\mbox{in}\hspace{0.2cm}\rn\times \mathbb{R}_+\\
			c(0)=c_0,n(0)=n_0, v(0)=v_0, u(0)=u_0&\mbox{in}\hspace{0.2cm}\rn
		\end{cases}  
	\end{align}
	where $c,n,u,p$ have the same meaning as before, $v$ is the concentration of chemical attractant,   $\kappa\geq 0$ represents the decay rate of the attractant and $\Psi$ is an external force acting on the fluid. For $c=0$ and $u=0$, \eqref{Sd-eq} reduces to the classical parabolic-parabolic Keller-Segel system of chemotaxis \cite{Ke-Se}
	\begin{align}\label{Keller-Seg}  
		\begin{cases}
			\partial_tn-\Delta n=-\dv (n\nabla v)\\
			\partial_tv-\Delta v=n-\kappa v
		\end{cases}  
		\mbox{in}\hspace{0.2cm}\Omega\times (0,\infty).
	\end{align}
	This equation has been widely studied in the literature, see for instance \cite{BBTW,CC,FPr,Fe-Ve,Horstmann,NSY} just to cite a few for well-posedness and blow-up results in bounded and unbounded domains. It is not clear whether the total mass of the initial density $\int_{\Omega}n_0dy$, as in the elliptic-parabolic case \cite{BCM}, plays a critical role.   Comparing  \eqref{Sd-eq} to \eqref{main-eq}, one sees that the new equation with unknown $v$ has no proper scaling if $\kappa \neq 0$. However, we can take advantage of our earlier analysis and the scaling property inherited from the case $\kappa=0$, i.e. $v_{\delta}(x,t)=v(\delta^2t,\delta x)$ to make the choice of $v_0$ in $BMO(\rn)$. In fact, due to the sign of $\kappa$ the semigroup associated to the operator $L_{k}:=-\Delta+\kappa$ shares similar properties with the heat semigroup. In order to state our results in this case we consider for $0<T\leq \infty$, the function space  \[\Z_{T}=\big\{[c,n,v,u]: c\in \X_{1,T},\hspace{0.1cm} n\in \X_{2,T},\hspace{0.1cm} v\in \X_{1,T}, \hspace{0.1cm}u\in \X_{3,T}\big\}\]
	equipped with the norm \[\big\|[c,n,v,u]\big\|_{\Z_T}=\|c\|_{\X_{1,T}}+\|n_0\|_{\X_{2,T}}+\|v\|_{\X_{1,T}}+\|u\|_{\X_{3,T}}.\]
	The local and global well-posedness result pertaining to \eqref{Sd-eq} is given in the next theorems.
	\begin{theorem}[Local well-posedness]\label{thm:LWP1}Assume $\Psi\in M_{2,N-2}(\rn)$, $N\geq 2$ and $d_0\in UC(\rn)$.  For all $[c_0,n_0,v_0,u_0]\in\overline{UC(\rn)}^{L^{\infty}(\rn)}\times \overline{V\mathscr{L}_{2,N-2}^{-1}}(\rn)\times VMO(\rn)\times\overline{VMO^{-1}}(\rn)$ with $\nabla\cdot u_0=0 $, there exist $\delta_0>0$, $T:=T(\delta_0)>0$ and a unique local solution of Eqs. \ref{Sd-eq} such that $[c,n,v-\widetilde{v_{0\kappa}},u]\in (\varGamma_{\delta_0}+\X_{1,T^2})\times \X_{2,T^2}\times \X_{1,T^2}\times \X_{3,T^2}$ provided $\Psi$ is small enough in $M_{2,N-2}(\rn)$. For $\kappa>0$, $\widetilde{v_{0\kappa}}$ denotes the $L_{\kappa}$-caloric extension of $v_0$.  
	\end{theorem}
	
	Next, a companion theorem shows that the maximal existence time in Theorem \ref{thm:LWP1} can be taken infinite, $T=\infty$ if one allows the data to be sufficiently small in the appropriate space. To state this result, we will need a new framework $\Z_0$ defined as \[\Z_0=\big\{[c_0,n_0,v_0,u_0]: c_0\in L^{\infty}(\rn),\hspace{0.051cm} n_0\in \Y(\rn),\hspace{0.051cm} v_0\in BMO(\rn),\hspace{0.1cm} u_0\in BMO^{-1}(\rn)\big\}\]
	and equipped with the norm 
	\[\big\|[c_0,n_0,v_0,u_0]\big\|_{\Z_0}=\|c_0\|_{L^{\infty}(\rn)}+\|n_0\|_{\Y(\rn)}+\|v_0\|_{BMO(\rn)}+\|u_0\|_{BMO^{-1}(\rn)}.\]
	
	\begin{theorem}\label{thm:GWP2} Assume $N\geq 2$. Eqs. \ref{Sd-eq} is globally well-posed.  There exist $\varepsilon>0$ and $\vartheta:=\vartheta(\varepsilon)>0$ with the following property. For any  $[c_0,n_0,v_0,u_0]\in \Z_0$ with $\nabla\cdot u_0=0$ and  $\Psi\in M_{2,N-2}(\rn)$ satisfying 
		$\big\|[c_0,n_0,v_0,u_0]\big\|_{\Z_0}+ \|\Psi\|_{M_{2,N-2}(\rn)} <\varepsilon$,
		there exists a mild solution $[c,n,v,u]$ of Eqs. \ref{Sd-eq}. This solution is unique in the set
		\[B^{\Z}_{2\vartheta}:=\big\{[c,n,v,u]\in\Z: \big\|[c,n,v,u]-[0,0,\widetilde{v}_{\kappa},0]\big\|_{\Z}\leq 2\vartheta\big\}.\] Moreover, the following uniqueness criterion holds: Let $[c_1,n_1,v_1,u_1]$, $[c_2,n_2,v_2,u_2]$ be two global mild solutions of Eqs. \ref{Sd-eq} in $L^{\infty}_{loc}((0,\infty);L^{\infty}(\rn))$ with the same initial data. If the  condition  
		\begin{align}\label{limit-cond-1}
			\lim_{T\rightarrow 0}\big\|[c_1,n_1,v_1,u_1]\big\|_{\Z_{T}}=0,\quad \lim_{T\rightarrow 0}\big\|[c_2,n_2,v_2,u_2]\big\|_{\Z_{T}}=0
		\end{align}
		is satisfied, then $[c_1,n_1,v_1,u_1]=[c_2,n_2,v_2,u_2]$ on $\rn\times [0,\infty)$.
	\end{theorem}
	From the two previous theorems, we deduce that the Keller-Segel system \eqref{Keller-Seg} is locally well-posed for initial data in $\overline{V\mathscr{L}_{2,N-2}^{-1}}(\rn)\times VMO(\rn)$ and globally well-posed whenever $[n_0,v_0]\in \mathscr{L}_{2,N-2}^{-1}(\rn)\times BMO(\rn)$ with $\|n_0\|_{\mathscr{L}_{2,N-2}^{-1}(\rn)}+\|u_0\|_{BMO(\rn)}$ sufficiently small. Let us mention that (local) solutions of  \eqref{Sd-eq} also enjoy the mass conservation property. 
	
	Before further comments on our main results, we would like to discuss their optimality.
	\subsection*{Optimality of main results} That $\X_{3}$ and thus $BMO^{-1}(\rn)$ are the largest critical spaces for mild solutions and the data relatively to the Navier-Stokes equations have already been justified in \cite{ADT,CG}. A detailed discussion about critical spaces can be found in \cite{L}. In the sequel we show that the setting where the equations \eqref{main-eq}$_1$-\eqref{main-eq}$_2$ are solved is also  end-point. Let us denote by $E^i$ and $E^i_0\subset \mathcal{S}'(\rn)$, $i=1,2$ Banach function spaces defined on $\rn\times (0,\infty)$ and $\rn$ which are critical for \eqref{main-eq}$_1$ and \eqref{main-eq}$_2$, respectively so that
	\begin{align}
		\label{scc}\|c_0\|_{E_0^1}=\|c_0(\delta \cdot-z)\|_{E_0^1},\quad \|c\|_{E^1}=\|c(\delta \cdot-z,\delta^2\cdot)\|_{E^1}\\
		\label{scn}\|n_0\|_{E_0^2}=\|\delta^2n_0(\delta \cdot-z)\|_{E_0^2},\quad \|n\|_{E^2}=\|\delta^2 n(\delta \cdot-z,\delta^2\cdot)\|_{E^2}     
	\end{align}
	for all $\delta>0$ and $z\in \rn$. To give a sense to the nonlinear term in \eqref{main-eq}$_2$ (and thus the bilinear form $B_2$, for instance in a distributional sense) we need at least that $n,\nabla c\in L^2_{loc}(\rn\times (0,\infty))$ since $u$ already satisfies this condition. The first requirement translates into the continuous embedding $E^2\subset L^2_{loc}(\rn\times (0,\infty))$ so that there exists $M>0$ with
	\begin{equation*}
		\bigg(\int_{0}^1\int_{B_1}|n|^2dyds\bigg)^{1/2}\leq M\|n\|_{E^2}    
	\end{equation*}
	for all $n\in E^2$ where $B_1\subset \rn$ is the unit ball while the second implies that
	\begin{equation*}
		\bigg(\int_{0}^1\int_{B_1}|\nabla c|^2dyds\bigg)^{1/2}\leq M\|c\|_{E^1}    
	\end{equation*}
	for all $c\in E^1$. The second identity in \eqref{scn} (resp. \eqref{scc}) clearly implies that
	\begin{equation}\label{Carlnorm-n}
		\sup_{x\in \rn,r>0}\bigg(|B_r(x)|^{\frac{2-N}{N}}\int_{0}^{r^2}\int_{B_r(x)}|n(y,s)|^2dyds\bigg)^{1/2}\leq M\|n\|_{E^2}\hspace{0.2cm}\mbox{for all}\hspace{0.2cm}n\in E^2    
	\end{equation}
	and (resp.)
	\begin{equation}\label{Carlnorm-c}
		\sup_{x\in \rn,r>0}\bigg(|B_r(x)|^{-1}\int_{0}^{r^2}\int_{B_r(x)}|\nabla c(y,s)|^2dyds\bigg)^{1/2}\leq M\|c\|_{E^1}\hspace{0.2cm}\mbox{for all}\hspace{0.2cm}c\in E^1.    
	\end{equation}
	Now, given an initial data $n_0\in E_0^2$, we look for a solution $n\in E^2$ and observe that under the conditions
	\begin{equation}\label{smoothing-E^2}
		\|Sn_0\|_{E^2}\leq M\|n_0\|_{E^2_0}     \hspace{0.2cm}\mbox{for all}\hspace{0.2cm}n_0\in E^2_0
	\end{equation}
	and the continuity of the bilinear form $B_2$ (see Section \ref{sec:3} below for its definition) on $E^2\times \X_3$, a Banach fixed point argument can be applied to produce such a unique global mild solution provided $n_0$ is small in the norm of $E^2_0$. Given \eqref{smoothing-E^2}, one easily deduces from  \eqref{Carlnorm-n} the estimate   
	\begin{equation}\label{Carl-Sn0}
		\sup_{x\in \rn,r>0}\bigg(|B_r(x)|^{\frac{2-N}{N}}\int_{0}^{r^2}\int_{B_r(x)}|Sn_0|^2dyds\bigg)^{1/2}\leq M\|n_0\|_{E^2_0}
	\end{equation}
	for each $n_0\in E^2_0$. The finiteness of the expression on the left hand side in \eqref{Carl-Sn0} is equivalent to $n_0\in \Y(\rn)$ by Definition \ref{def:M-1}. Hence, $E^2_0\subset \Y(\rn)$ and  $E^2$ is continuously embedded in the space of functions such that the left hand side in \eqref{Carlnorm-n} is finite. Moving on, the  extra condition allowing one to make sense of the first nonlinear term in \eqref{main-eq}$_1$ is the local boundedness of $c$. Because of scaling invariance, we see that $E^1\subset L^{\infty}(\rn\times (0,\infty))$ must hold true. Combining this with \eqref{Carlnorm-c}, we find that
	\begin{equation}\label{Carlinftynorm-c}
		\sup_{x\in \rn,r>0}\bigg(|B_r(x)|^{-1}\int_{0}^{r^2}\int_{B_r(x)}|\nabla c|^2dyds\bigg)^{1/2}+\|c\|_{L^{\infty}(\rn\times (0,\infty))}\leq M\|c\|_{E^1}\hspace{0.2cm}\mbox{for all}\hspace{0.2cm}c\in E^1.    
	\end{equation}
	Likewise, the existence of a unique $c$ relies on two properties: the continuity of the bilinear operator $B_1$ on $E^1\times E^2$, $B_2$ on $E^1\times \X_3$ and the validity of the estimate $\|Sc_0\|_{E^1}\leq M\|c_0\|_{E^1_0}$
	for all $c_0\in E^1_0$. The latter and \eqref{Carlinftynorm-c} lead to
	\begin{equation}\label{Carlinftynorm-c0}
		\sup_{x\in \rn,r>0}\bigg(|B_r(x)|^{-1}\int_{0}^{r^2}\int_{B_r(x)}|\nabla Sc_0|^2dyds\bigg)^{1/2}+\|Sc_0\|_{L^{\infty}(\rn\times (0,\infty))}\leq M\|c_0\|_{E^1_0}    
	\end{equation}
	for any $c_0\in E^1_0$. The finiteness of the first term on the left hand side of the above inequality is equivalent to $c_0\in BMO(\rn)$. However, this will violate the finiteness of the second term while both terms are finite if $c_0\in L^{\infty}(\rn)$. On the other hand, it is clear that $E^1$ is continuously embedded in the space of functions $c$ for which the semi-norm of their gradient as in \eqref{Carlinftynorm-c} is finite.  A close reasoning shows that the results pertaining to \eqref{Sd-eq} are also sharp.
	
	\begin{remark}\label{opti}
		As a direct consequence of Theorem \ref{thm:LWP}, one deduces the existence of global in time solutions for initial concentrations  which are small $L^{\infty}$-perturbations of constants in $\mathbb{R}$. Also, the global solutions constructed depend continuously on the initial data. This is a simple consequence of the argument used in the proofs which mainly relies on the contraction mapping principle. Whether or not the smallness assumption on the force in the above stated theorems can be removed is an interesting question but which is not addressed in this paper.   
	\end{remark}
	
	\begin{remark}\label{lem:embedding-BM}Comparing our main results with earlier findings, one merely requires the initial data $c_0$ to belong to $L^{\infty}(\rn)$ together with a suitable smallness condition and no assumption on its gradient is needed.
		Moreover, for $2\leq p<\infty$, $N>2$ and $0\leq \lambda<N$, due to the embeddings (see the Appendix for the proofs)
		\begin{align}\label{embedding-BM-Y_Besov}
			\mathcal{N}^{-2s}_{p,\lambda,\infty}(\rn)\subset \Y(\rn)\subset \dot{B}^{-2}_{\infty,\infty}(\rn);\hspace{0.2cm}s=1+\dfrac{\lambda-N}{2p},\hspace{0.1cm}(N-\lambda)/2\leq p<N-\lambda    
		\end{align}
		\begin{align}\label{embedding-BM-BMO-1}
			\mathcal{N}^{-s}_{p,\lambda,\infty}(\rn)\subset BMO^{-1}(\rn);\quad s=1-\frac{N-\lambda}{p},\hspace{0.2cm}p>N-\lambda
		\end{align}
		our initial data class in Theorem \ref{thm:GWP} is larger than those considered in \cite{KMS,YFS}. Likewise, in Theorem \ref{thm:GWP2}, the initial concentration of chemical attractant is taken in $BMO(\rn)$ and no extra requirement on its first order gradient is necessary unlike in the articles \cite{FP,KY}. In fact, their initial data classes are contained in ours when the dimension is larger or equal to $3$. This plainly shows that our global existence results encompasses all those which have been cited before. In 2D, however, the initial concentration $n_0$ is taken in a smaller class $\dot{B}^{-1}_{22}(\mathbb{R}^2)$ but gives rise to a much natural functional setting. It is worth pointing out that our local well-posedness results (Theorems \ref{thm:LWP} and \ref{thm:LWP1}) are derived under much weaker regularity assumptions as compared to those obtained for instance in \cite{CKL,Zha} and related works therein. Finally, because $M_{2,N-2}(\rn)\subset M_{p,N-p}(\rn)$ for $1\leq p<2$, our well-posedness theorems are not optimal with respect to the external force, see \cite{FP,YFS}.    
	\end{remark}
	\begin{remark}[Self-similar solutions]
		From the embedding $L^{N/2,\infty}(\rn)\subset \Y(\rn)$, one sees that $\Y(\rn)$ contains homogeneous distributions of degree $-2$. Thus, if $[c_0,n_0,u_0]$ is homogeneous of degree $0$, $-2$ and $-1$, respectively and $\big\|[c_0,n_0,u_0]\big\|_{\X_0}$ is sufficiently small, then as a by-product of Theorem \ref{thm:GWP}, there exits a unique (forward self-similar) solution $[c,n,u]$ satisfying
		\begin{equation}\label{self-similar}
			n(x,t)=\delta^2n(\delta x,\delta ^2t), \hspace{0.12cm}c(x,t)=c(\delta  x,\delta ^2t),\hspace{0.12cm}u(x,t)=\delta  u(\delta  x,\delta ^2t)\hspace{0.52cm}\mbox{for all} \hspace{0.21cm}\delta >0.   
		\end{equation}
		provided $\Phi\in \mathcal{S}'(\rn)$, $\nabla \Phi\in M_{2,N-2}(\rn)$ is homogeneous of degree $-1$ with $\|\nabla\Phi\|_{M_{2,N-2}(\rn)}$ small enough. 
		A similar conclusion persists if in the double chemotaxis Navier-Stokes equation \eqref{Sd-eq}, one takes $\kappa=0$, $v_0$ and $\Psi$ homogeneous of degree $0$ and $-1$, respectively. 
	\end{remark}
To end this section, we point out potential applications of our results.
	\begin{itemize}
		\item (Regularity of mild solutions)  A natural question to ask is: what is the space-time regularity of mild solutions to \eqref{main-eq} and \eqref{Sd-eq} arising from small initial data in $\X_0$ and $\Z_0$ respectively? A naive approach to this question will be to perform a similar type of analysis in the analogue of our solution spaces but this time for higher order derivatives of arbitrary order. Similar ideas involving space derivatives appeared earlier in \cite{GPS} where the authors proved that global mild solutions of the Navier-Stokes equations subject to small $BMO^{-1}$ initial data  are space-analytic.   
		\item (Stability) Given the optimal nature of our results, one may ask whether those solutions are stable. Let $(c_0,n_0,u_0)$ be an initial data giving rise to a global solution of either \eqref{main-eq} or \eqref{Sd-eq}. By stability here, we mean to ask the question of whether a small perturbation of the initial data in either topology still generates a unique global solution. This question can be reformulated in terms of the openness of the set of such $(c_0,n_0,u_0)$. An answer to this question, in the affirmative case may use some intrinsic properties of our initial data class. See for instance \cite{ADT,L} in the context of the Navier-Stokes equations.     
	\end{itemize}

	\section{Preliminaries and auxiliary results}\label{sec:3}
	In this section, we collect key estimates for the homogeneous and the inhomogeneous heat equation.  For a suitable function $f$ (e.g. smooth and compactly supported), 
	denote by $Sf$ the operator \[Sf(x,t)=e^{\Delta t}f(x)=(g_t\ast f)(x)\]
	where $g_t(x)=g(x,t)=\dfrac{e^{\frac{|x|^2}{4t}}}{(4\pi t)^{\frac{N}{2}}}$. Then $Sf$ solves the heat equation $(\partial_t-\Delta)u=0$ in $\rn\times (0,\infty)$, $u(0)=f$ on $\rn$. 
	\begin{lemma}\label{lem:smooth-effect}
		Assume $N>2$. Let $0<R\leq \infty$. The operator $S$ maps $L^{\infty}(\rn)$ to $\X_{1,R^2}$, $\mathscr{L}^{-1}_{2,N-\lambda;R}(\rn)$ to $\X_{2,R^2}$ and $BMO^{-1}_{R}(\rn)$ to $\X_{3,R^2}$ continuously. If $N=2$, then $Sf\in \X_2$ whenever $f\in \dot{B}^{-1}_{2,2}(\mathbb{R}^2)$ and there exists $C>0$ independent of $f$ such that
		\begin{align}
			\|Sf\|_{\X_2}\leq C\|f\|_{\dot{B}^{-1}_{2,2}(\mathbb{R}^2)}.    
		\end{align}
	\end{lemma}
	\begin{proof}
		From the Carleson measure characterization of $BMO$ (see e.g. \cite{St}) it is well-known that
		\[\sup_{x\in \rn,r>0}|B(x,r)|^{-1}\int_{0}^{r^2}\int_{B(x,r)}|\nabla Sf(x,t)|^2dxdt\approx \|f\|_{BMO(\rn)}.\] If $f\in BMO_R(\rn)$ for some $0<R\leq\infty$, then an analogue of the above inequality holds where the supremum on the left-hand side is taken over all balls of radius $R$ and smaller. Thus, from the continuous embedding $L^{\infty}(\rn)\subset BMO_R(\rn)$ one gets the estimate 
		\[\sup_{x\in \rn,0<r\leq R}|B(x,r)|^{-1}\int_{0}^{r^2}\int_{B(x,r)}|\nabla Sf(x,t)|^2dxdt\leq C\|f\|_{L^{\infty}(\rn)}\]
		for any $f\in L^{\infty}(\rn).$
		On the other hand, the estimate \[\displaystyle\sup_{0<t\leq R^2}\|Sf(t)\|_{L^{\infty}(\rn)}+\sup_{0<t\leq R^2}t^{1/2}\|\nabla Sf(t)\|_{L^{\infty}(\rn)}\leq C\|f\|_{L^{\infty}(\rn)}\] follows straight from the smoothing effect of the heat semigroup. 
		The boundedness property of $S$ from $\mathscr{L}^{-1}_{2,N-2;R}(\rn)$ to $\X_{2,R^2}$ is a consequence of  \eqref{embedding-BM-Y_Besov} and the Carleson's characterization of $\mathscr{L}^{-1}_{2,N-2;R}(\rn)$ (see Definition \ref{def:M-1} and Lemma \ref{lem:M^{-1}-charac} in the Appendix) while the proof of the third statement can be found in \cite{L}. In the case $N=2$,  we shall show that 
		\begin{align}\label{2D-lin1}
			\sup_{t>0}t\|u(t)\|_{L^{\infty}(\mathbb{R}^2)}\leq C\|f\|_{\dot{B}^{-1}_{2,2}(\mathbb{R}^2)} 
		\end{align}
		and
		\begin{align}\label{2D-lin2}
			\|u\|_{L^{2}(\mathbb{R}^2\times (0,\infty))}\leq C\|f\|_{\dot{B}^{-1}_{2,2}(\mathbb{R}^2)}. 
		\end{align}
		While the latter estimate follows from the caloric characterization of Besov spaces \cite{Tr}, the former is established as follows. For $x\in \rn$ and $0<s<t/2$, one may use semigroup properties and H\"{o}lder's inequality to get 
		\begin{align*}
			|e^{t\Delta}f(x)|&= \big|e^{(t-s)\Delta}e^{s\Delta}f(x)\big|\\
			&\leq C\bigg(\dfrac{1}{t}\int^{t/2}_0\int_{\mathbb{R}^2} g(x-y,t-s)|e^{s\Delta}f(y)|^2dyds\bigg)^{1/2}\\
			&\leq C\bigg(\dfrac{1}{t}\int^{t/2}_0\int_{\mathbb{R}^2} (t-s)^{-1}e^{-\frac{|x-y|^2}{4(t-s)}}|e^{s\Delta}f(y)|^2dyds\bigg)^{1/2}\\
			&\leq Ct^{-1}\bigg(\int^{t/2}_0\int_{\mathbb{R}^2} e^{-\frac{|x-y|^2}{4(t-s)}}|e^{s\Delta}f(y)|^2dyds\bigg)^{1/2}\\
			&\leq  Ct^{-1}\|e^{s\Delta}f\|_{L^{2}(\mathbb{R}^2\times (0,\infty))}\\
			&\leq Ct^{-1}\|f\|_{\dot{B}^{-1}_{2,2}(\mathbb{R}^2)}.
		\end{align*}
		Passing to the supremum on both sides over all $t\in (0,\infty)$ yields the desired bound.
	\end{proof}
	Let $d_0 \in UC(\rn)$, for any $\varepsilon_0>0$, there exists $\delta_0>0$ such that for all $x,y\in \rn$ with $|x-y|<\delta_0$, we have $|d_0(x)-d_0(y)|\leq \varepsilon_0.$
	\begin{lemma}\label{lem:UC-pert}
		Let $d_0 \in UC(\rn)$ and set $\varGamma_{\delta_0}=e^{\delta_0^2\Delta}d_0$. Then the following estimates hold.
		\begin{align}
			\|\varGamma_{\delta_0}\|_{L^{\infty}(\rn)}\leq C\\
			\|\varGamma_{\delta_0}-d_0\|_{L^{\infty}(\rn)}+\delta_0\|\nabla\varGamma_{\delta_0}\|_{L^{\infty}(\rn)}+\delta_0^2\|\nabla^2\varGamma_{\delta_0}\|_{L^{\infty}(\rn)}\leq C\varepsilon_0
		\end{align}
		for some constant $C>0$.
	\end{lemma}
	\begin{proof}
		We refer the reader to  \cite{KL}.
	\end{proof}
	\subsection{Bilinear estimates}
	Let $R_j=\partial_j(-\Delta)^{-1/2}$, $j=1,...,N$ denote the Riesz transform and $\textbf{P}$ be the Leray–Hopf projection operator onto divergence-free vector fields defined (component-wise) on $L^2(\rn)$ by 
	\begin{equation*}
		\textbf{P}_{kj}=\delta_{kj}+R_kR_j,\quad j,k=1,...,N
	\end{equation*}
	where $\delta_{kj}$ is the Kronecker symbol. Applying $\textbf{P}$ to the Navier-Stokes equations
	in \eqref{main-eq}, the resulting equations can be recast into the following integral system
	\begin{align}\label{Duhamel-CNS}
		\begin{cases}
			\displaystyle c=e^{t\Delta}c_0-\int_0^te^{(t-s)\Delta}(c n+u\cdot\nabla c)(\cdot,s)ds\\
			\displaystyle n=e^{t\Delta}n_0-\int_0^te^{(t-s)\Delta}\nabla\cdot(n\nabla c+ nu)(\cdot,s)ds\\
			\displaystyle u=e^{t\Delta}u_0-\int_0^te^{(t-s)\Delta}\textbf{P}\nabla\cdot (u\otimes u)(\cdot,s)ds-\int_0^te^{(t-s)\Delta} \textbf{P}(n\nabla \Phi)(\cdot,s)ds.
		\end{cases}    
	\end{align}
	Define the linear map 
	\begin{align*}
		\mathscr{L}_{\Phi}(n)&=\int_0^te^{(t-s)\Delta}\textbf{P}(n\nabla\Phi)(\cdot,s)ds
	\end{align*}
	and the bilinear maps 
	\begin{align*}
		B_1(w,n)&=\int_0^te^{(t-s)\Delta} (wn)(\cdot,s)ds,\\
		B_2(n,w)&=\int_0^te^{(t-s)\Delta}\nabla\cdot (n w)(\cdot,s)ds,\\
		B_3(u,w)&=\int_0^te^{(t-s)\Delta}\textbf{P}\nabla\cdot (u\otimes w)(\cdot,s)ds
	\end{align*}
	whenever the integrals are well-defined. The next lemma establishes the continuity properties of these maps in targeted functions spaces. 
	\begin{lemma}\label{lem:bilinear-est}Let $N\geq  2$ and $0<T\leq \infty$. Assume that $\Phi\in \mathcal{S}'(\rn)$ with $\nabla\Phi\in M_{2,N-2}(\rn)$. The linear operator  $\mathscr{L}_{\Phi}(n):\X_{2,T}\rightarrow \X_{3,T}$ continuously and the bilinear operators $B_j(\cdot,\cdot)$, $j=1,2,3$ are such that
		\[B_1:\X_{1,T}\times \X_{2,T}\rightarrow \X_{1,T},\hspace{0.12cm} B_1:\X_{3,T}\times \X_{3,T}\rightarrow \X_{1,T},\hspace{0.12cm} B_2:\X_{3,T}\times \X_{3,T}\rightarrow \X_{1,T}\]
		and  $B_3:\X_{3,T}\times \X_{3,T}\rightarrow \X_{3,T}$ continuously. Moreover, there exists $C_j>0$, $j=1,\cdots,5$ such that
		\begin{align}
			\|\mathscr{L}_{\Phi}(n)\|_{\X_{3,T}}&\leq C_1\|n\|_{\X_{2,T}}\|\nabla\Phi\|_{M_{2,N-2}(\rn)}\quad \mbox{for all}\hspace{0.32cm} n\in \X_{2,T}\label{1}\\
			\|B_1(w,n)\|_{\X_{1,T}}&\leq C_2\|w\|_{\X_{1,T}}\|n\|_{\X_{2,T}}\quad \mbox{for all}\hspace{0.32cm} w\in \X_{1,T}\hspace{.1cm}\mbox{and}\hspace{.1cm}n\in \X_{2,T}\label{2}\\  
			\|B_1(u,w)\|_{\X_{1,T}}&\leq C_3\|u\|_{\X_{3,T}}\|w\|_{\X_{3,T}}\quad \mbox{for all}\hspace{0.32cm} u,w\in \X_{3,T}\\  
			\|B_2(n,w)\|_{\X_{2,T}}&\leq C_4\|n\|_{\X_{2,T}}\|w\|_{\X_{3,T}}\quad \mbox{for all}\hspace{0.32cm} n\in \X_{2,T}\hspace{.1cm}\mbox{and}\hspace{.1cm}w\in \X_{3,T}\label{3}\\
			\|B_3(v,w)\|_{\X_{3,T}}&\leq C_5\|v\|_{\X_{3,T}}\|w\|_{\X_{3,T}}\quad \mbox{for all}\hspace{0.32cm} v,w\in \X_{3,T}.    
		\end{align}
	\end{lemma}
	Let $f$ be a locally integrable function and  $\mathcal{M}f$ its  uncentered maximal function  defined as 
	\[\mathcal{M}f(x)=\sup_{B \ni x}|B|^{-1}\int_{B}|f(y)|dy.\]
	Let $p\in (1,\infty)$ with $1/p+1/p'=1$. A nonnegative measurable function $\mu$ on $\rn$ belongs to the Muckenhoupt weight class $A_p(\rn)=A_p$ if 
	\begin{equation*}
		[\mu]_{A_p}:=\sup_{B\subset \rn}\bigg(\fint_{B}\mu(x)dx\bigg)\bigg(\fint_{B}\mu(x)^{-p'/p}dx\bigg)^{p/p'}<\infty.  
	\end{equation*}
	Given a weight $\mu\in A_p$ and denoting $L^p_{\mu}(\rn)=L^p(\rn,\mu dx)$, it is well-known (see e.g. \cite{GR}) that $\mu\in A_p$ if and only if $\mathcal{M}$ is bounded on $L^p_{\mu}(\rn)$. Given a non-negative measurable function $h$ on $(0,\infty)$ and $\alpha,\beta\in (0,1)$, define the fractional integral operator 
	\begin{equation*}
		E(h)(s)=\int^s_0(s-\sigma)^{\alpha-1}\sigma^{-\beta}h(\sigma)d\sigma. 
	\end{equation*}
	The next lemma establishes the boundedness properties of $E$ between  weighted-Lebesgue spaces. 
	\begin{lemma}\label{lem:bound-E}
		Let $\alpha,\beta\in (0,1)$ and $p\in (1,\infty)$ such that $-1/p<\alpha-\beta<1/p'$ holds. Then $E$ maps $L_{\nu}^p((0,\infty))$ continuously into $L^p((0,\infty))$ where  $\nu(s)=s^{(\alpha-\beta)p}$, $s>0$. In particular, $E$ is bounded on $L^p((0,\infty))$ $($including $p=\infty$$)$ if $\alpha=\beta$.
	\end{lemma}
	\begin{proof}
		The proof of the lemma relies on the following pointwise estimate for the operator $E$:  for $0<\alpha,\beta<1$, there exists $C:=C(\alpha,\beta)>0$ such that
		\begin{equation}\label{pointwise-bd}
			|E(h)(s)|\leq Cs^{\alpha-\beta}\mathcal{M}h(s).    
		\end{equation}
		The proof of \eqref{pointwise-bd} is direct, we refer the reader to \cite{MYZ} for details. An immediate consequence of the latter is that the mapping properties of $E$ may be deduced from those of $\mathcal{M}$. Since the function $\nu(s)=s^{(\alpha-\beta)p}$  is an $A_p$-weight under the restriction $-1/p<\alpha-\beta<1/p'$, one has 
		\begin{align}
			\|E(h)\|_{L^p((0,\infty))}\leq C\|\mathcal{M}h\|_{L^p_{\nu}((0,\infty))}\leq C\|h\|_{L^p_{\nu}((0,\infty))}.    
		\end{align}
	\end{proof}
	\begin{proof}[Proof of Lemma \ref{lem:bilinear-est}] We estimate $B_1$, $B_2$ and $\mathscr{L}_{\Phi}(n)$ in 3 steps respectively. The required bound on the bilinear map $B_3$ is known and can be found for instance in \cite{KT}.  
		\begin{step1}[Estimates on $B_1$]\normalfont We prove that $B_1(\cdot,\cdot)$ is continuous from $\X_{1,T}\times \X_{2,T}$ to $\X_{1,T}$. Mimicking the same steps, we similarly show that $B_1(\cdot,\cdot):\X_{3,T}\times \X_{3,T}\rightarrow \X_{1,T}$ is continuous. The details of the latter case are therefore omitted. 
			Let $w\in \X_{1,T}$ and $n\in \X_{2,T}$. For  $(x,t)\in \rn\times (0,T]$ write
			\begin{align*}
				B_1(w,n)(x,t)&=\int_0^t\int_{\rn}g(x-y,t-s)(wn)(y,s)dyds:=B_{11}(w,n)(x,t)+B_{12}(w,n)(x,t) 
			\end{align*}
			where 
			\begin{align*}
				B_{11}(w,n)(x,t)&=\int_0^{t/2}\int_{\rn}g(x-y,t-s)(wn)(y,s)dyds,\\
				B_{12}(w,n)(x,t)&=\int_{t/2}^t\int_{\rn}g(x-y,t-s)(wn)(y,s)dyds.    
			\end{align*}
			Let $B^c_{r}(x)$ denote the complement of the Euclidean ball $B_{r}(x)=B(x,r)$ with center at $x\in \rn$ and radius $r>0$ and further make the decomposition 
			\begin{align*}
				B_{11}(w,n)(x,t)&=\int_0^{t/2}\bigg(\int_{B_{2\sqrt{t}}(x)}+\int_{B^c_{2\sqrt{t}}(x)}\bigg)g(x-y,t-s)[wn](y,s)dyds\\
				&=B^1_{11}(w,n)(x,t)+B^2_{11}(w,n)(x,t).
			\end{align*}
			Using H\"older's inequality, one gets 
			\begin{align*}
				|B^1_{11}(w,n)(x,t)|&\leq \int_0^{t/2}\int_{B_{2\sqrt{t}}(x)}g(x-y,t-s)|(nw)(y,s)|dyds\\ 
				&\leq C\sup_{0<t\leq T}\|w(t)\|_{L^{\infty}(\rn)}\|g\|_{L^{2}(B_{2\sqrt{t}}(0)\times [t/2,t])}\bigg(\int_0^{t/2}\int_{B_{2\sqrt{t}}(x)}|n(y,s)|^2dyds\bigg)^{\frac{1}{2}}\\
				&\leq C\|w\|_{\X_{1,T}}t^{\frac{2-N}{4}}\bigg(\int_0^{t/2}\int_{B_{2\sqrt{t}}(x)}|n(y,s)|^2dyds\bigg)^{\frac{1}{2}}\leq C\|w\|_{\X_{1,T}}\|n\|_{\X_{2,T}}.
			\end{align*}
			On the other hand, if one sets $A_j(x)=B_{2(j+1)\sqrt{t}}(x)\setminus B_{2j\sqrt{t}}(x)$, then it follows that 
			\begin{align*}
				|B^2_{11}(w,n)(x,t)|&\leq \int_0^{t/2}\int_{B^c_{2\sqrt{t}}(x)}g(x-y,t-s)|(nw)(y,s)|dyds\\
				&\leq C\sup_{0<t\leq T}\|w(t)\|_{L^{\infty}(\rn)}\sum_{j=1}^{\infty}\int_0^{t/2}\int_{A_j(x)}\dfrac{e^{-\frac{|x-y|^2}{4(t-s)}}}{(t-s)^{N/2}}|n(y,s)|dyds\\
				&\leq C\|w\|_{\X_{1,T}}\sum_{j=1}^{\infty}e^{-2j^2}\sum_{z\in A_j(x)\cap \sqrt{t}\mathbb{Z}^N}\int_0^{t/2}\int_{B_{\sqrt{t}}(z)}(t-s)^{-N/2}|n(y,s)|dyds\\
				&\leq C\|w\|_{\X_{1,T}}\sum_{j=1}^{\infty}e^{-2j^2}\sum_{z\in A_j(x)\cap \sqrt{t}\mathbb{Z}^N}t^{\frac{2-N}{4}}\bigg(\int_0^{t/2}\int_{B_{\sqrt{t}}(z)}|n(y,s)|^2dyds\bigg)^{1/2}\\
				&\leq C\|w\|_{\X_{1,T}}\bigg(\sum_{j=1}^{\infty}j^{N-1}e^{-2j^2}\bigg)\sup_{z\in \rn}\bigg(t^{2-N}\int_0^{t/2}\int_{B_{\sqrt{t}}(z)}|n(y,s)|^2dyds\bigg)^{1/2}\\
				&\leq C\|w\|_{\X_{1,T}}\|n\|_{\X_{2,T}}.
			\end{align*}
			Since $g\in L^1(\rn)$, it holds that 
			\begin{align*}
				|B_{12}(w,n)(x,t)|&=\int_{t/2}^t\int_{\rn}g(x-y,t-s)|(wn)(y,s)|dyds\\
				&\leq C\sup_{0<t\leq T} \|w(t)\|_{L^{\infty}(\rn)}\sup_{0<t\leq T}t\|n(t)\|_{L^{\infty}(\rn)}t^{-1}\int_{0}^{t/2}\|g(\cdot,s)\|_{L^{1}(\rn)}ds\\
				&\leq C\|w\|_{\X_{1,T}}\|n\|_{\X_{2,T}}.
			\end{align*}
			Moving on, we prove the pointwise gradient bound on $B_1$. For any $x\in \rn$, $0<t\leq T$ and $k_1(x,t)=\nabla_x g(x,t)$ we have
			\begin{align*}
				\nabla & B_1(w,n)(x,t)\\
				&=\int^{t}_0\int_{\rn}k_1(x-y,t-s)(wn)(y,s)dyds\\
				&=\int^{\frac{t}{2}}_0\int_{\rn}k_1(x-y,t-s)(wn)(y,s)dyds+\int^{t}_{\frac{t}{2}}\int_{\rn}k_1(x-y,t-s)(wn)(y,s)dyds\\
				&=B^1_1(w,n)(x,t)+B^2_1(w,n)(x,t).
			\end{align*}
			We estimate each of these terms using the fact that  $k_1(x,t)=-t^{-1}xg(x,t)$  (recall $g$ is the heat kernel). Indeed,
			\begin{align*}
				|B^2_1(w,n)(x,t)|&\leq C\|w\|_{\X_{1}}\|n\|_{\X_{2}}\int^{t}_{t/2}\int_{\rn}s^{-1}|k_1(x-y,t-s)|dyds\\
				&\leq C\|w\|_{\X_{1}}\|n\|_{\X_{2}}\int^{t}_{t/2}\int_{\rn}\dfrac{|x-y|}{s(t-s)}g(x-y,t-s)dyds\\
				&\leq Ct^{-1/2}\|w\|_{\X_{1,T}}\|n\|_{\X_{2,T}}.
			\end{align*}
			Now, note that $\|k_1\|_{L^{2}(B_{2\sqrt{t}}(x)\times [t/2,t])}\leq Ct^{-\frac{N}{4}}$ so that by arguing as above, we find
			\begin{align*}
				&|B^1_1(w,n)(x,t)|\\
				&\leq \int_0^{t/2}\int_{B_{2\sqrt{t}}(x)}|k_1(x-y,t-s)||(nw)(y,s)|dyds+\\
				&\hspace{5cm}\int_0^{t/2}\int_{B^c_{2\sqrt{t}}(x)}|k_1(x-y,t-s)||(nw)(y,s)|dyds\\
				&\leq C\sup_{0<t\leq T}\|w(t)\|_{L^{\infty}(\rn)}\|k_1\|_{L^{2}(B_{2\sqrt{t}}(0)\times [t/2,t])}\|n\|_{L^{2}(B_{2\sqrt{t}}(x)\times (0,t/2])}+\\
				&\hspace{2.5cm}C\sup_{0<t\leq T}\|w(t)\|_{L^{\infty}(\rn)}\sum_{j=1}^{\infty}\int_0^{t/2}\int_{A_j(x)}|k_1(x-y,t-s)||n(y,s)|dyds\\
				&\leq C\|w\|_{\X_{1,T}}t^{-\frac{N}{4}}\bigg(\int_0^{t/2}\int_{B_{\sqrt{t}}(x)}|n(y,s)|^2dyds\bigg)^{1/2}+\\
				&\hspace{1.6cm} C\|w\|_{\X_{1,T}}\sum_{j=1}^{\infty}(j+1)e^{-2j^2}\sum_{z\in A_j(x)\cap \sqrt{t}\mathbb{Z}^N}t^{\frac{-N}{4}}\bigg(\int_0^{t/2}\int_{B_{\sqrt{t}}(z)}|n(y,s)|^2dyds\bigg)^{\frac{1}{2}}\\
				&\leq Ct^{-\frac{1}{2}}\|w\|_{\X_{1,T}}\|n\|_{\X_{2,T}}+\\
				&\hspace{1cm} Ct^{-\frac{1}{2}}\|w\|_{\X_{1,T}}\sum_{j=1}^{\infty}(j+1)j^{N-1}e^{-2j^2}\sup_{z\in \rn}\bigg(t^{-\frac{2-N}{2}}\int_0^{t/2}\int_{B_{\sqrt{t}}(z)}|n(y,s)|^2dyds\bigg)^{\frac{1}{2}}\\
				&\leq Ct^{-\frac{1}{2}}\|w\|_{\X_{1,T}}\|n\|_{\X_{2,T}}.
			\end{align*}
			To estimate the $L^2$-gradient norm, write 
			\begin{align*}
				&|B(x,t)|^{-1}\int_0^{t^2}\int_{B(x,t)}|\nabla B_1(n,w)(y,s)|^2dyds\\
				&=|B(x,t)|^{-1}\int_0^{t^2}\int_{B(x,t)}|I_1(n,w)(y,s)|^2dyds+|B(x,t)|^{-1}\int_0^{t^2}\int_{B(x,t)}|I_2(n,w)(y,s)|^2dyds\\
				&=I(w,n)(x,t)+II(w,n)(x,t),\quad 0<t\leq T
			\end{align*}
			where  
			\begin{equation*}
				I_1(w,n)(y,s)=\int^s_0\int_{\rn}k_1(y-z,s-\sigma)(wn\textbf{1}_{B(x,2t)})(z,\sigma)dzd\sigma  
			\end{equation*}
			and 
			\begin{equation*}
				I_2(w,n)(y,s)=\int^s_0\int_{\rn}k_1(y-z,s-\sigma)(wn\textbf{1}_{B^c(x,2t)})(z,\sigma)dzd\sigma.  
			\end{equation*}
			Since $|k_1(y,s)|\leq C(|y|^2+s)^{-\frac{N+1}{2}}$ for $y\in \rn, s>0,$  one may use Young's convolution inequality to obtain
			\begin{align*}
				\|I_1(w,n)(\cdot,s)\|_{L^2(\rn)}&\leq C\int^s_0(s-\sigma)^{-1/2}\|wn\textbf{1}_{B_{2t}(x)}(\cdot,\sigma)\|_{L^2(\rn)}d\sigma\\
				&\leq C\|w\|_{\X_{1,T}}s^{1/2}\int_0^s(s-\sigma)^{-1/2}\sigma^{-1/2}\|(n\textbf{1}_{B_{2t}(x)})(\sigma)\|_{L^2(\rn)}d\sigma.
			\end{align*}
			Thus, by  Lemma \ref{lem:bound-E} applied with $p=2$; $\alpha=\beta=\frac{1}{2}$, $I(w,n)(x,t)$ may be estimated as follows
			\begin{align*}
				\big| I(w,n)(x,t)\big|&\leq |B(x,t)|^{-1}\int_0^{t^2}\|I_1(\cdot,s)\|^2_{L^2(\rn)}ds\\
				&\leq C\|w\|^2_{\X_{1,T}}|B(x,t)|^{-1}t^2\int^{t^2}_0\|n\textbf{1}_{B_{2t}(x)}(\cdot,s)\|^2_{L^2(\rn)}ds\\
				&\leq C\|w\|^2_{\X_{1,T}}|B(x,t)|^{\frac{2}{N}-1}\int_0^{t^2}\int_{B_{2t}(x)} |n(y,s)|^2dyds\\
				&\leq C\|w\|^2_{\X_{1,T}}\|n\|^2_{\X_{2,T}}.
			\end{align*}  
			Remark that if $z\in B^c(x,2t)$ and $y\in B(x,t)$, then $|y-z|\geq \dfrac{1}{2}|x-z|$. Thus, for $s\leq t^2<T$, one has
			\begin{align*}
				|I_2(w,n)(y,s)|&\leq \int^s_0\int_{|x-z|\geq 2t}|k_1(y-z,s-\sigma)||(wn)(z,\sigma)|dzd\sigma\\    
				&\leq C\int^s_0\int_{|x-z|\geq 2t}\dfrac{|(wn)(z,\sigma)|dzd\sigma }{(|y-z|+(s-\sigma)^{1/2})^{N+1}}\\
				&\leq C\|w\|_{\X_{1,T}}\int^s_0\int_{|x-z|\geq 2t}|y-z|^{-(N+1)}|n(z,\sigma)|dzd\sigma\\
				&\leq C\|w\|_{\X_{1,T}}\int^{t^2}_0\sum_{j=1}^{\infty}\int_{B(x,2(j+1)t)\setminus B(x,2jt)}|x-z|^{-(N+1)}|n(z,\sigma)|dzd\sigma\\
				&\leq C\|w\|_{\X_{1,T}}t^{-(N+1)}\sum_{j=1}^{\infty}j^{-(N+1)}\sum_{\substack{q\in t\mathbb{Z}^N\\ q\in B(x,2(j+1)t)\setminus B(x,2jt)}}\int^{t^2}_0\int_{B(q,t)}|n(z,\sigma)|dzd\sigma\\
				&\leq C\|w\|_{\X_{1,T}}t^{-1}\bigg(\sum_{j=1}^{\infty}j^{-2}\bigg)\sup_{q\in \rn}\bigg(t^{2-N}\int^{t^2}_0\int_{B(q,t)}|n(z,\sigma)|^2dzd\sigma\bigg)^{\frac{1}{2}}\\
				&\leq Ct^{-1}\|w\|_{\X_{1,T}}\|n\|_{\X_{2,T}}.
			\end{align*}
			This estimate directly gives the desired bound for $II(w,n)$, namely  \[\sup_{x\in \rn,0<t\leq \sqrt{T}}\big|II(w,n)(x,t)\big|\leq C\|w\|_{\X_{1,T}}\|n\|_{\X_{2,T}}.\]  
		\end{step1}
		
		\begin{step2}[The bounds on $B_2$] \normalfont Let $n\in \X_{2,T}$ and $w\in \X_{3,T}$, we want to show that 
			\begin{equation}\label{B2}
				\|B_2(n,w)\|_{\X_{2,T}}\leq C\|n\|_{\X_{2,T}}\|w\|_{\X_{3,T}}.    
			\end{equation}
			We first estimate the norm $\displaystyle \sup_{x\in \rn,0<t\leq \sqrt{T}}[B_2(n,w)]^{1/2}_{x,t}$ where
			\begin{equation*}
				[B_2(n,w)]_{x,t}:=|B(x,t)|^{2/N-1}\int_0^{t^2}\int_{B(x,t)}|B_2(n,w)|^2dyds.
			\end{equation*}
			To this end, split $B_2(n,w)(y,s)$ into two parts
			\begin{equation*}
				B_{21}(n,w)(y,s)=\int^s_0\int_{\rn}\nabla g(y-z,s-\sigma)\cdot(nw\textbf{1}_{B_{2t}(x)})(z,\sigma)dzd\sigma,  
			\end{equation*}
			and 
			\begin{equation*}
				B_{22}(n,w)(y,s)=\int^s_0\int_{\rn}\nabla g(y-z,s-\sigma)\cdot(nw\textbf{1}_{B^c_{2t}(x)})(z,\sigma)dzd\sigma.  
			\end{equation*}
			Arguing as before, we can show that the following inequality is valid, namely 
			\begin{align*}
				|B_{22}(n,w)(x,t)|\leq Ct^{-2}\|n\|_{\X_{2,T}}\|w\|_{\X_{3,T}}     
			\end{align*}
			from which we immediately get
			\begin{equation}\label{B_2-nonlocal}
				\sup_{x\in \rn,0<t\leq \sqrt{T}}[B_{22}(n,w)]_{x,t}\leq C\|n\|^2_{\X_{2,T}}\|w\|^2_{\X_{3,T}}.    
			\end{equation}
			Next, using Young's convolution inequality and Lemma \ref{lem:bound-E} with $\alpha=\beta=1/2$ it follows that
			\begin{align*}
				\|B_{22}&(n,w)(\cdot,s)\|_{L^2(\rn)}\\
				&\leq C\int^s_0(s-\sigma)^{-1/2}\|wn\textbf{1}_{B_{2t}(x)}(\cdot,\sigma)\|_{L^2(\rn)}d\sigma\\
				&\leq C\sup_{0<t\leq T}t^{1/2}\|w(\cdot,t)\|_{L^{\infty}(\rn)}\int_0^s(s-\sigma)^{-1/2}\sigma^{-1/2}\|(n\textbf{1}_{B_{2t}(x)})(\sigma)\|_{L^2(\rn)}d\sigma
			\end{align*}
			and 
			\begin{align}\label{B_2local}
				\nonumber[B_{21}(n,w)]_{x,t}&=|B(x,t)|^{2/N-1}\int_0^{t^2}\int_{B(x,t)}|B_{21}(n,w)(y,s)|^2dyds\\
				\nonumber&\leq C |B(x,t)|^{2/N-1}\int_0^{t^2}\|wn\textbf{1}_{B_{2t}(x)}(\cdot,s)\|^2_{L^2(\rn)}ds\\
				\nonumber&\leq C\|w\|^2_{\X_{3,T}}|B(x,t)|^{2/N-1}\big\|\|(n\textbf{1}_{B_{2t}(x)})\|_{L^2(\rn)}(\cdot)\big\|_{L^2(0,t^2]}\\
				\nonumber&\leq C\|w\|^2_{\X_{3,T}}|B(x,t)|^{\frac{2}{N}-1}\int_0^{t^2}\int_{B(x,2t)} |n(y,s)|^2dyds\\
				&\leq C\|w\|^2_{\X_{3,T}}\|n\|^2_{\X_{2,T}}.
			\end{align}  
			Combining \eqref{B_2-nonlocal} and \eqref{B_2local}, we get the desired bound. Next, we show that 
			\begin{equation}\label{pointwise-bd-B_2}
				\sup_{0<t\leq T}t\|B_2(n,w)\|_{L^{\infty}(\rn)}\leq C\|n\|_{\X_{2,T}}\|w\|_{\X_{3,T}}.    
			\end{equation}
			Once again, we make the decomposition 
			\begin{align*}
				B_2(n,w)(x,t)&=\int_0^t\int_{\rn}\nabla g(x-y,t-s)\cdot(nw)(y,s)dyds:=B^1_{2}(n,w)(x,t)+B^2_{2}(n,w)(x,t)
			\end{align*}
			where 
			\begin{align*}
				B^1_{2}(n,w)(x,t)&=\int_0^{t/2}\int_{\rn}\nabla g(x-y,t-s)\cdot(nw)(y,s)dyds,\\
				B^2_{2}(n,w)(x,t)&=\int_{t/2}^t\int_{\rn}\nabla g(x-y,t-s)\cdot(nw)(y,s)dyds.    
			\end{align*}
			To bound $B^2_{2}(n,w)$, we use the kernel decay bound and proceed as follows 
			\begin{align*}
				|B^2_{2}(n,w)(x,t)|&\leq C\|w\|_{\X_{3,T}}\|n\|_{\X_{2,T}}\int^{t}_{t/2}\int_{\rn}s^{-3/2}|\nabla g(x-y,t-s)|dyds\\
				&\leq C\|w\|_{\X_{3,T}}\|n\|_{\X_{2,T}}\int^{t}_{t/2}\int_{\rn}s^{-3/2}\dfrac{|x-y|}{(t-s)}g(x-y,t-s)dyds\\
				&\leq Ct^{-1}\|w\|_{\X_{3,T}}\|n\|_{\X_{2,T}}.
			\end{align*}
			To estimate $B^1_{2}(n,w)$, we decompose it as a sum of two integrals and use H\"{o}lder's inequality: 
			\begin{align*}
				|&B^1_{2}(n,w)(x,t)|\\
				&\leq \int_0^{t/2}\int_{B_{2\sqrt{t}}(x)}|\nabla g(x-y,t-s)||(nw)(y,s)|dyds+\\
				&\hspace{3cm}\int_0^{t/2}\int_{B^c_{2\sqrt{t}}(x)}|\nabla g(x-y,t-s)||(nw)(y,s)|dyds\\
				&\leq Ct^{-\frac{N+1}{2}}\bigg(\int_0^{t/2}\int_{B(x,2\sqrt{t})}|n(y,s)|^2dyds\bigg)^{\frac{1}{2}}\bigg(\int^{t/2}_{0}\int_{B(x,2\sqrt{t})}|w(y,s)|^2dyds\bigg)^{\frac{1}{2}}+\\
				&\hspace{3.7cm}Ct^{-\frac{N+1}{2}}\sum_{j=1}^{\infty}(j+1)e^{-2j^2}\int_0^{t/2}\int_{A_j(x)}|(nw)(y,s)|dyds\\
				&\leq Ct^{-1}\|w\|_{\X_{3,T}}\|n\|_{\X_{1,T}}+\\ &\hspace{0cm}Ct^{-\frac{N+1}{2}}\bigg(\sum_{j=1}^{\infty}(j+1)j^{N-1}e^{-2j^2}\bigg)\bigg(\int_0^{t/2}\int_{B_{\sqrt{t}}(z)}|n(y,s)|^2dyds\bigg)^{\frac{1}{2}}\bigg(\int_0^{t/2}\int_{B_{\sqrt{t}}(z)}|w(y,s)|^2dyds\bigg)^{\frac{1}{2}}\\
				&\leq Ct^{-1}\|w\|_{\X_{3,T}}\|n\|_{\X_{2,T}}.
			\end{align*}
			This shows \eqref{pointwise-bd-B_2} and finishes Step 2. 
		\end{step2}
		
		\begin{step3}[\textit{The bounds on $\mathscr{L}_{\Phi}(n)$}] \normalfont Here we show that $\mathscr{L}_{\Phi}$ is continuous from $ \X_{2,T}$ to $\X_{3,T}$ for any tempered distribution $\Phi$ with $\nabla\Phi\in M_{2,N-2}(\rn)$. The operator $e^{t\Delta}\textbf{P}$ is an integral operator whose kernel is given by the Oseen kernel $k_2(t)$ which satisfies the polynomial decay bound (see for instance \cite[Chapter 11]{L})
			\begin{equation}\label{Oseen}
				t^{|\alpha|/2}|\partial^{\alpha} k_2(x,t)|\leq Ct^{-N/2}(1+t^{-1/2}|x|)^{-N-|\alpha|}\hspace{0.5cm}\mbox{for all}\hspace{0.2cm} \alpha\in \mathbb{N}^N, \hspace{0.1cm}x\in \rn\hspace{0.1cm}\mbox{and}\hspace{0.1cm}t>0.    
			\end{equation}
			Set
			\begin{equation*}
				[\mathscr{L}_{\Phi}n]_{x,t}:=\bigg(|B(x,t)|^{-1}\int_0^{t^2}\int_{B(x,t)}|(\mathscr{L}_{\Phi}n)(y,s)|^2dyds\bigg)^{1/2}.    
			\end{equation*}
			We primarily show that $\displaystyle\sup_{x\in \rn,0<t\leq T^{1/2}}[\mathscr{L}_{\Phi}n]_{x,t}\leq C\|n\|_{\X_{2,T}}\|\nabla\Phi\|_{M_{2,N-2}(\rn)}$.
			Define
			\begin{equation*}
				\mathscr{L}^1_{\Phi}n(y,s)=\int^s_0\int_{\rn}k_2(y-z,s-\sigma)[(n\textbf{1}_{B_{2t}(x)})(\sigma)\nabla\Phi)](z)dzd\sigma,  
			\end{equation*}
			\begin{equation*}
				\mathscr{L}^2_{\Phi}n(y,s)=\int^s_0\int_{\rn}k_2(y-z,s-\sigma)[(n\textbf{1}_{B^c_{2t}(x)})(\sigma)\nabla\Phi)(z)]dzd\sigma.  
			\end{equation*}
			For $0<s\leq t^2<T$ and $y\in B(x,t)$, we have 
			\begin{align*}
				|\mathscr{L}^2_{\Phi}(n)(y,s)|&\leq \int^s_0\int_{|x-z|\geq 2t}|k_2(y-z,s-\sigma)||n(z,\sigma)\nabla\Phi(z)|dzd\sigma\\    
				&\leq C\int^s_0\int_{|x-z|\geq 2t}\dfrac{(s-\sigma)^{-N/2}|n(z,\sigma)\nabla\Phi(z)|}{(1+|y-z|(s-\sigma)^{-1/2})^{N}}dzd\sigma\\ 
				&\leq C\int^s_0\sum_{j=1}^{\infty}\int_{B(x,2^{j+1}t)\setminus B(x,2^{j}t)}\dfrac{(s-\sigma)^{-N/2}|n(z,\sigma)\nabla\Phi(z)|}{(1+|x-z|(s-\sigma)^{-1/2})^{N}}dzd\sigma\\
				&\leq Ct^{1-N}\sum_{j=1}^{\infty}2^{-jN}\bigg(\int^{t^2}_0\int_{B(x,2^{j+1}t)}|n(z,\sigma)|^2dzd\sigma\bigg)^{\frac{1}{2}}\bigg(\int_{B(x,2^{j+1}t)}|\nabla\Phi(z)|^2dz\bigg)^{\frac{1}{2}}\\
				&\leq Ct^{-1}\bigg(\sum_{j=1}^{\infty}4^{-j}\bigg)\|n\|_{\X_{2,T}}\|\nabla\Phi\|_{M_{2,N-2}(\rn)}\leq Ct^{-1}\|n\|_{\X_{2,T}}\|\nabla\Phi\|_{M_{2,N-2}(\rn)}. 
			\end{align*}
			where H\"{o}lder's inequality was used to obtain the estimate before the last. Hence, 
			\begin{align*}
				\sup_{x\in \rn,0<t\leq \sqrt{T}} [\mathscr{L}^2_{\Phi}n]_{x,t}\leq C\|n\|_{\X_{2,T}}\|\nabla\Phi\|_{M_{2,N-2}(\rn)}.   
			\end{align*}
			On the other hand, let $0<\eta<1/2$, $1<\theta<\dfrac{N}{N+2\eta-1}$. Take $1<\theta_0<2$ such that $\dfrac{1}{\theta}+\dfrac{1}{\theta_0}=\dfrac{3}{2}$. Then by Young's inequality  we find that
			\begin{align}\label{pointwise-B31}
				\nonumber\big\|[\mathscr{L}^1_{\Phi}n](s)\big\|_{L^2(\rn)}&\leq C\int_0^{s}(s-\sigma)^{\frac{N}{2}(1/\theta-1)}\|(n\textbf{1}_{B(x,2t)})(\sigma)\nabla\Phi\|_{L^{\theta_0}(\rn)}d\sigma\\    
				\nonumber&\leq Ct^{2\eta}\int_0^{s}(s-\sigma)^{\frac{N}{2}(\frac{1}{\theta}-1)}\sigma^{-\eta}\|(n\textbf{1}_{B(x,2t)})(\sigma)\nabla\Phi\|_{L^{\theta_0}(\rn)}d\sigma.
			\end{align}
			This implies (in view of Lemma \ref{lem:bound-E} with $\alpha=1+\frac{N}{2}(\frac{1}{\theta}-1)$, $\beta=\eta$ and $p=2$) that
			\begin{align*}
				[\mathscr{L}^1_{\Phi}(n)]^2_{x,t}&\leq |B(x,t)|^{-1}\int^{t^2}_0\|\mathscr{L}^1_{\Phi}(n)(\cdot,s)\|^2_{L^2(\rn)}ds\\
				&\leq Ct^{4\eta}|B(x,t)|^{-1}\int^{t^2}_0s^{2+N(1/\theta-1)-2\eta}\|n(s)\textbf{1}_{B_{2t}(x)}\nabla\Phi\|^2_{L^{\theta_0}(\rn)}ds\\
				&\leq C\big(\sup_{0<t\leq T}t\|n(t)\|_{L^{\infty}(\rn)}\big)^2t^{4\eta-N}\|\nabla\Phi\|^2_{L^{\theta_0}(B(x,2t))}\int_0^{t^2}s^{N(1/\theta-1)-2\eta}ds\\
				&\leq C\|n\|^2_{\X_{2,T}}t^{4\eta-N+N(2/\theta_0-1)}\|\nabla\Phi\|^2_{L^{2}(B(x,2t))}\int_0^{t^2}s^{N(1/\theta-1)-2\eta}ds\\
				&\leq C\|\nabla\Phi\|^2_{M_{2,N-2}(\rn)}\|n\|^2_{\X_{2,T}}.
			\end{align*}  
			Finally, we need to prove the estimate 
			\begin{equation}\label{lastbd-B3}
				\sup_{0<t\leq T}t^{1/2}\|[\mathscr{L}_{\Phi}n](\cdot,t)\|_{L^{\infty}(\rn)}\leq C\|n\|_{\X_{2,T}}\|\nabla\Phi\|_{M_{2,N-2}(\rn)}.    
			\end{equation}
			Making use of \eqref{Oseen} and H\"{o}lder's inequality we have
			\begin{align*}
				\big|\mathscr{L}^1_{\Phi}(n)(x,t)\big|&\leq \int_0^{t/2}\int_{B_{2\sqrt{t}}(x)}\dfrac{|n(y,s)\nabla\Phi(y)|dyds}{(|x-y|+(t-s)^{1/2})^N}+\int_0^{t/2}\int_{B^c_{2\sqrt{t}}(x)}\dfrac{|n(y,s)\nabla\Phi(y)|dyds}{(|x-y|+(t-s)^{1/2})^N}\\
				&\leq Ct^{-N/2}\int_0^{t/2}\bigg(\int_{B_{2\sqrt{t}}(x)}|n(y,s)|^2dy\bigg)^{\frac{1}{2}}ds\|\nabla\Phi\|_{L^2(B_{2\sqrt{t}}(x))}+\\
				&\hspace{4cm}C\sum_{j=1}^{\infty}\int_0^{t/2}\int_{B(x,2^{j+1}\sqrt{t})\setminus B(x,2^j\sqrt{t})}\dfrac{|n(y,s)\nabla\Phi(y)|dyds}{(|x-y|+(t-s)^{1/2})^N}\\
				&\leq Ct^{\frac{1-N}{2}}\bigg(\int_0^{t/2}\int_{B_{2\sqrt{t}}(x)}|n(y,s)|^2dyds\bigg)^{\frac{1}{2}}\|\nabla\Phi\|_{L^2(B_{2\sqrt{t}}(x))}+\\
				&\hspace{1.cm} Ct^{\frac{1-N}{2}}\sum_{j=1}^{\infty}2^{-jN}\bigg(\int_0^{t/2}\int_{B_{2^{j+1}\sqrt{t}}(x)}|n(y,s)|^2dyds\bigg)^{\frac{1}{2}}\|\nabla\Phi\|_{L^2(B_{2^{j+1}\sqrt{t}}(x))}.
			\end{align*}
			This clearly implies that
			\begin{align*}
				\sup_{0<t\leq T} t^{\frac{1}{2}}\big\|\mathscr{L}^1_{\Phi}(n)(t)\big\|_{L^{\infty}(\rn)}   &\leq C\|\nabla\Phi\|_{M_{2,N-2}(\rn)}\|n\|_{\X_{2,T}}.
			\end{align*}
			On the other hand, using \eqref{Oseen} with $|\alpha|=1$ we find that
			\begin{align*}
				\big|[k_2(\cdot,t)\ast (n(s)\nabla \Phi)](x)\big|&=\bigg|\int_{\rn}k_2(x-y,t)(n(s)\nabla \Phi)(y)dy\bigg|\\
				&\leq \int_0^{\infty}\bigg|\dfrac{dk_2(r,t)}{dr}\bigg|\bigg(\int_{B_r(x)}|n(y,s)\nabla \Phi(y)|dy\bigg)dr\\
				&\leq Ct^{-N/2-1/2}\int_{0}^{\infty}(1+t^{-1/2}r)^{-(N+1)}r^{N/2}\|n(s)\nabla \Phi\|_{L^2(B_r(x))}dr\\
				&\leq Ct^{-1/2}\|n(s)\nabla \Phi\|_{M_{2,N-2}(\rn)}\int_{0}^{\infty}\dfrac{r^{N-1}}{(1+r)^{N+1}}dr.
			\end{align*}
			Thus,
			\begin{align*}
				|\mathscr{L}^2_{\Phi}(n)(x,t)|&\leq \bigg|\int^{t}_{t/2}e^{(t-s)\Delta}\textbf{P}(n\nabla\Phi)(y,s)ds\bigg|\\
				&\leq \int^{t}_{t/2}\bigg\|e^{(t-s)\Delta}\textbf{P}(n\nabla\Phi)(\cdot,s)\bigg\|_{L^{\infty}(\rn)}ds\\
				&\leq C\int^{t}_{t/2}(t-s)^{-\frac{1}{2}}\|n(s,\cdot)\nabla\Phi\|_{M_{2,N-2}(\rn)}ds\\
				&\leq C\sup_{0<t\leq T}t\|n(t)\|_{L^{\infty}(\rn)}\|\nabla\Phi\|_{M_{2,N-2}(\rn)}\int^{t}_{t/2}(t-s)^{-\frac{1}{2}}s^{-1}ds\\
				&\leq Ct^{-\frac{1}{2}}\|n\|_{\X_{2,T}}\|\nabla\Phi\|_{M_{2,N-2}(\rn)}.
			\end{align*}
		\end{step3}
		This gives \eqref{lastbd-B3}, concludes Step 3 and thus the proof of Lemma \ref{lem:bilinear-est}.
	\end{proof}
	Now we turn to the proof of Theorem \ref{thm:GWP2}. Reformulate  \eqref{Sd-eq} into the following system whose solutions are referred to as mild solutions  
	\begin{align}\label{Duhamel-D-CNS}
		\begin{cases}
			\displaystyle c=e^{t\Delta}c_0-\int_0^te^{(t-s)\Delta}(c n+u\cdot\nabla c)(\cdot,s)ds\\
			\displaystyle n=e^{t\Delta}n_0-\int_0^te^{(t-s)\Delta}\nabla\cdot(un+n\nabla c+ n\nabla v)(\cdot,s)ds\\
			\displaystyle v=e^{-\kappa t}e^{t\Delta}v_0-\int_0^te^{-\kappa(t-s)}e^{(t-s)\Delta} (u\cdot \nabla v)(\cdot,s)ds-\int_0^te^{-\kappa(t-s)}e^{(t-s)\Delta}n(\cdot,s)ds\\
			\displaystyle u=e^{t\Delta}u_0-\int_0^te^{(t-s)\Delta}\textbf{P}\nabla\cdot (u\otimes u)(\cdot,s)ds-\int_0^te^{(t-s)\Delta} \textbf{P}(n \Psi)(\cdot,s)ds.
		\end{cases}    
	\end{align}
	In order to prove the well-posedness of Syst. \ref{Duhamel-D-CNS}, one needs in addition to Lemma \ref{lem:bilinear-est} another auxiliary result about the mapping properties of  the linear operator  
	$\mathscr{L}$ and the bilinear map $B_5$ respectively given by  
	\begin{equation*}
		\mathscr{L}n(t)=\int_0^te^{-\kappa(t-s)}e^{(t-s)\Delta}n(\cdot,s)ds, \hspace{0.2cm}t>0    
	\end{equation*}
	and 
	\begin{equation*}
		B_4(u,v)(t)=\int_0^te^{-\kappa(t-s)}e^{(t-s)\Delta}(u\cdot \nabla v)(\cdot,s)ds, \hspace{0.2cm}t>0,\hspace{0.2cm}\kappa>0.    
	\end{equation*}
	\begin{lemma}\label{lem:DCNS} Let $\kappa>0$ and $0<T\leq \infty$. The linear operator $\mathscr{L}$ is continuous from $\X_{2,T}$ to $\X_{1,T}$ and  there exists $C>0$ such that 
		\begin{align}\label{B5}
			\|\mathscr{L}n\|_{\X_{1,T}}\leq C\|n\|_{\X_{2,T}}   
		\end{align}
		for any $n\in \X_{2,T}$. Moreover, if $v\in \X_{1,T}$ and $u \in \X_{3,T}$, then $B_4(u,v)\in \X_{1,T}$ and  
		\begin{align}\label{}
			\|B_4(u,v)\|_{\X_{1,T}}\leq C\|u\|_{\X_{3,T}}\|v\|_{\X_{1,T}}.    
		\end{align}
	\end{lemma}
	This Lemma is proved in a similar fashion as before. We include the details for the sake of completeness.  
	\begin{proof}
		Let $n\in \X_{2,T}$, we have 
		\begin{align*}
			\mathscr{L}n(x,t)&=\bigg(\int_0^{t/2}+\int_{t/2}^t\bigg)e^{-\kappa(t-s)}e^{(t-s)\Delta}n(y,s)dyds:=J_1(x,t)+J_2(x,t).   
		\end{align*}
		Using only the property $g(\cdot,t)\in L^1(\rn)$ for every $t>0$, we bound $J_2$ as follows:
		\begin{align*}
			|J_2(x,t)|&=\int_{t/2}^te^{-\kappa(t-s)}\int_{\rn}g(x-y,t-s)|n(y,s)|dyds\\
			&\leq C\sup_{t>0}t\|n(t)\|_{L^{\infty}(\rn)}\int_{t/2}^ts^{-1}e^{-\kappa(t-s)}ds\leq C\|n\|_{\X_{2,T}}.   
		\end{align*}
		Regarding $J_1$, we further split it into two terms, $J_{1}=J_{11}+J_{12}$ with   
		\begin{align*}
			J_{11}(x,t)= \int_0^{t/2}\int_{B_{2\sqrt{t}}(x)}e^{-\kappa(t-s)}g(x-y,t-s)n(y,s)dyds,  
		\end{align*}
		\begin{align*}
			J_{12}(x,t)&= \int_0^{t/2}\int_{B^c_{2\sqrt{t}}(x)}e^{-\kappa(t-s)}g(x-y,t-s)n(y,s)dyds.    
		\end{align*}
		By H\"older's inequality, 
		\begin{align*}
			|J_{11}(x,t)|&\leq C\big\|e^{-\kappa\cdot }g(\cdot,\cdot)\big\|_{L^{2}(B_{2\sqrt{t}}(0)\times [t/2,t])}\bigg(\int_0^{t/2}\int_{B_{2\sqrt{t}}(x)}|n(y,s)|^2dyds\bigg)^{\frac{1}{2}}\\
			&\leq Ct^{-N/4+1/2}\bigg(\int_0^{t/2}\int_{B_{2\sqrt{t}}(x)}|n(y,s)|^2dyds\bigg)^{\frac{1}{2}}\leq C\|n\|_{\X_{2,T}}.
		\end{align*}
		With $A_j(x)$ as previously defined, we have 
		\begin{align*}
			|J_{12}(x,t)|&\leq \int_0^{t/2}\int_{B^c_{2\sqrt{t}}(x)}e^{-\kappa(t-s)}g(x-y,t-s)|n(y,s)|dyds\\
			&\leq C\sum_{j=1}^{\infty}e^{-\kappa(t-s)}e^{-2j^2}\sum_{ z\in A_j(x)\cap \sqrt{t}\mathbb{Z}^N}\int_0^{t/2}\int_{B_{\sqrt{t}}(z)}e^{-\kappa(t-s)}(t-s)^{-n/2}|n(y,s)|dyds\\
			&\leq C\sum_{j=1}^{\infty}e^{-2j^2}\sum_{z\in A_j(x)\cap \sqrt{t}\mathbb{Z}^N}t^{\frac{2-N}{4}}e^{-\frac{\kappa}{2t}}\bigg(\int_0^{t/2}\int_{B_{\sqrt{t}}(z)}|n(y,s)|^2dyds\bigg)^{1/2}\\
			&\leq C\bigg(\sum_{j=1}^{\infty}j^{N-1}e^{-2j^2}\bigg)\sup_{z\in \rn}\bigg(t^{2-N}\int_0^{t/2}\int_{B_{\sqrt{t}}(z)}|n(y,s)|^2dyds\bigg)^{1/2}\\
			&\leq C\|n\|_{\X_{2,T}}.
		\end{align*}
		The pointwise gradient estimate $\displaystyle \sup_{0<t\leq T}t^{\frac{1}{2}}\|\mathscr{L}n(\cdot,t)\|_{L^{\infty}(\rn)}\leq C\|n\|_{\X_{2,T}}$ follows the same steps using the decay of the kernel $k_1(t)=\nabla_xg(t)$. Next, we show the energy-type estimate
		\begin{align*}
			[\mathscr{L}n]_{Car}:=\sup_{x\in \rn,0<t\leq T}\bigg(|B(x,\sqrt{t})|^{-1}\int^t_0\int_{B(x,\sqrt{t})}|\nabla\mathscr{L}n(y,s)|^2dyds\bigg)^{1/2}\leq C\|n\|_{\X_{2,T}}.    
		\end{align*}
		Let $n\in \X_{2,T}$ and for  $(x,t)\in \rn\times (0,T]$ write (in the sense of distributions)
		\begin{align*}
			\nabla \mathscr{L}n(x,t)&=\int_0^s\int_{\rn}e^{-\kappa(s-\sigma)}k_1(y-z,s-\sigma)(\textbf{1}_{B(x,2\sqrt{t})}n)(z,\sigma)dzd\sigma+\\
			&\hspace{2cm}\int_0^s\int_{\rn}e^{-\kappa(s-\sigma)}k_1(y-z,s-\sigma)(1-\textbf{1}_{B(x,2\sqrt{t})})n(z,\sigma)dzd\sigma\\
			&:=I_{1}(n)(y,s)+I_{2}(n)(y,s). 
		\end{align*}
		As such, we have that
		\begin{align*}
			[\mathscr{L}n]_{Car}&=\sup_{x\in \rn,0<t\leq T}\bigg(|B(x,\sqrt{t})|^{-1}\int^t_0\int_{B(x,\sqrt{t})}|I_1(n)(y,s)|^2dyds\bigg)^{1/2}+\\
			&\hspace{3cm}\sup_{x\in \rn,0<t\leq T}\bigg(|B(x,\sqrt{t})|^{-1}\int^t_0\int_{B(x,\sqrt{t})}|I_2(n)(y,s)|^2dyds\bigg)^{1/2}\\
			&=M_1+M_2.    
		\end{align*}
		Using the following $L^2$-bound 
		\begin{align*}
			\|I_1(n)(\cdot,s)\|_{L^2(\rn)}&\leq C\int^s_0e^{-\kappa(s-\sigma)}(s-\sigma)^{-1/2}\|n\textbf{1}_{B(x,2\sqrt{t})}(\cdot,\sigma)\|_{L^2(\rn)}d\sigma,
		\end{align*}
		the bound on $M_1$ follows from  an application of Lemma \ref{lem:bound-E} with $p=2$, $\alpha=\beta=1/2$. Indeed,
		\begin{align*}
			M_1&\leq \sup_{x\in \rn,0<t\leq T}\bigg(|B(x,\sqrt{t})|^{-1}\int^t_0\|K_1(n)(\cdot,s)\|_{L^2(\rn)}^2ds\bigg)^{1/2}\\
			&\leq C\sup_{x\in \rn,0<t\leq T}\bigg(|B(x,\sqrt{t})|^{-1}t\int^t_0\|n\textbf{1}_{B(x,2\sqrt{t})}(s)\|^2_{L^2(\rn)}ds\bigg)^{1/2}\\
			&\leq C\sup_{x\in \rn,0<t\leq T}\bigg(|B(x,\sqrt{t})|^{-(1-2/N)}\int^t_0\int_{B(x,2\sqrt{t})}|n(y,s)|^2dyds\bigg)^{1/2}\\
			&\leq C\|n\|_{\X_{2,T}}.
		\end{align*}
		Also, for $y\in B(x,\sqrt{t})$ and $z\in \rn\setminus B(x,2\sqrt{t})$, we have $|y-z|\geq \frac{1}{2}|x-z|$. Thus  for $0<s\leq t$, 
		\begin{align*}
			|I_2(n)(y,s)|&\leq \int_0^s\int_{\rn\setminus B(x,2\sqrt{t})}e^{-\kappa(s-\sigma)}|k_1(y-z,s-\sigma)||n(z,\sigma)|dzd\sigma\\
			&\leq C\int_0^t\int_{|x-z|\geq 2\sqrt{t}}\dfrac{e^{-\kappa(s-\sigma)}}{(|x-y|+(s-\sigma)^{1/2})^{N+1}}|n(z,\sigma)|dzd\sigma\\
			&\leq C\int_0^t\int_{|x-z|\geq 2\sqrt{t}}\dfrac{|n(z,\sigma)|dzd\sigma}{|x-y|^{N+1}}.
		\end{align*}
		Performing a similar covering argument as before we obtain the estimate $M_2\leq C\|n\|_{\X_{2,T}}$. To conclude the proof we show that
		\begin{align}
			\|B_4(u,v)\|_{\X_{1,T}}\leq C\|u\|_{\X_{3,T}}\|v\|_{\X_{1,T}}    
		\end{align}
		for all $u\in \X_{3,T}$ and $v\in \X_{1,T}$.
		Put
		\begin{align*}
			N_{1}(u,v)(x,t)&=\int_0^{t/2}\int_{\rn}e^{-\kappa(t-s)} g(x-y,t-s)(u\cdot \nabla v)(y,s)dyds\\
			N_{2}(u,v)(x,t)&=\int_{t/2}^t\int_{\rn}e^{-\kappa(t-s)}g(x-y,t-s)(u\cdot \nabla v)(y,s)dyds.
		\end{align*}
		We have 
		\begin{align*}
			|N_{2}(u,v)(x,t)|&=\int_{t/2}^t\int_{\rn}e^{-\kappa(t-s)}g(x-y,t-s)|(u\cdot \nabla v)(y,s)|dyds\\
			&\leq C\sup_{0<t\leq T}t^{\frac{1}{2}}\|u(t)\|_{L^{\infty}(\rn)}\sup_{0<t\leq T}t^{\frac{1}{2}}\|\nabla v(t)\|_{L^{\infty}(\rn)}t^{-1}\int_{0}^{t/2}e^{-\kappa s}\|g(s)\|_{L^1(\rn)}ds\\
			&\leq C\|u\|_{\X_{3,T}}\|v\|_{\X_{1,T}}.
		\end{align*}
		We estimate $N_1$ using H\"older's inequality: 
		\begin{align*}
			|N_{1}(u,v)(x,t)|&=\int_0^{t/2}\int_{B_{2\sqrt{t}}(x)}e^{-\kappa(t-s)}g(x-y,t-s)|(u\cdot \nabla v)(y,s)|dyds+\\
			&\hspace{2cm}\int_0^{t/2}\int_{\rn\setminus B_{2\sqrt{t}}(x)}e^{-\kappa(t-s)}g(x-y,t-s)|(u\cdot \nabla v)(y,s)|dyds\\
			&\leq Ct^{-N/2}\bigg(\int_0^{t/2}\int_{B_{2\sqrt{t}}(x)}|u(y,s)|^2dyds\bigg)^{\frac{1}{2}}\bigg(\int_0^{t/2}\int_{B_{2\sqrt{t}}(x)}|\nabla v(y,s)|^2dyds\bigg)^{\frac{1}{2}}+\\
			& \hspace{4cm}C\sum_{j=1}^{\infty}\int_0^{t/2}\int_{A_j(x)}e^{-\kappa(t-s)}\dfrac{e^{-\frac{|x-y|^2}{4(t-s)}}}{(t-s)^{N/2}}|(u\cdot \nabla v)(y,s)|dyds\\
			&\leq C\|u\|_{\X_{3,T}}\|v\|_{\X_{1,T}}+\\
			&\hspace{1.2cm}C\sum_{j=1}^{\infty}e^{-2j^2}\sum_{z\in A_j(x)\cap \sqrt{t}\mathbb{Z}^N}\int_0^{t/2}\int_{B_{\sqrt{t}}(z)}e^{-\kappa (t-s)}(t-s)^{-\frac{N}{2}}|(u\cdot \nabla v)(y,s)|dyds\\
			&\leq C\|u\|_{\X_{1,T}}\|v\|_{\X_{1,T}}+\\
			&\hspace{2cm}C\sum_{j=1}^{\infty}e^{-2j^2}\sum_{z\in A_j(x)\cap \sqrt{t}\mathbb{Z}^N}t^{\frac{-N}{2}}e^{-\frac{\kappa t}{2}}\bigg(\int_0^{t/2}\int_{B_{\sqrt{t}}(z)}|u(y,s)|^2dyds\bigg)^{\frac{1}{2}}\cdot\\
			&\hspace{7.4cm}\bigg(\int_0^{t/2}\int_{B_{\sqrt{t}}(z)}|\nabla v(y,s)|^2dyds\bigg)^{\frac{1}{2}}\\
			&\leq C\|u\|_{\X_{3,T}}\|v\|_{\X_{1,T}}+\bigg(\sum_{j=1}^{\infty}j^{N-1}e^{-2j^2}\bigg)\|u\|_{\X_{3,T}}\|v\|_{\X_{1,T}}\leq C\|u\|_{\X_{3,T}}\|v\|_{\X_{1,T}}.
		\end{align*}
		The gradient estimate $\displaystyle\sup_{0<t\leq T}t^{1/2}\|\nabla B_4(u,v)\|_{L^{\infty}(\rn)}\leq C\|u\|_{\X_{3,T}}\|v\|_{\X_{1,T}}$ is obtained in a similar fashion. It remains to establish the $L^2$-gradient estimate 
		\begin{align}\label{Carl-B5}
			[B_4(u,v)]_{Car}=\sup_{x\in \rn,0<t\leq T}|B(x,\sqrt{t})|^{-1}\int_0^{t}\int_{B(x,\sqrt{t})}|\nabla B_4(u,v)(y,s)|^2dyds &\leq C\|u\|_{\X_{3,T}}\|v\|_{\X_{1,T}}.
		\end{align}
		Set  
		\begin{equation*}
			B_{41}(u,v)(y,s)=\int^s_0\int_{\rn}e^{-\kappa(t-s)}k_1(y-z,s-\sigma)[(u\textbf{1}_{B(x,2\sqrt{t})})\cdot \nabla v](z,\sigma)dzd\sigma,  
		\end{equation*}
		and 
		\begin{equation*}
			B_{42}(u,v)(y,s)=\int^s_0\int_{\rn}e^{-\kappa(t-s)}k_1(y-z,s-\sigma)[(u\textbf{1}_{\rn\setminus B(x,2\sqrt{t})})\cdot\nabla v](z,\sigma)dzd\sigma.  
		\end{equation*}
		Then, by Young's convolution inequality one has
		\begin{align*}
			\|B_{41}&(u,v)(s)\|_{L^2(\rn)}\\
			&\leq C\int^s_0e^{-\kappa(t-s)}\big\|[(u\textbf{1}_{B(x,2\sqrt{t})})\cdot\nabla v](\cdot,\sigma)\big\|_{L^2(\rn)}\\
			&\leq C\sup_{s\in (0, T)}s^{\frac{1}{2}}\|\nabla v(s)\|_{L^{\infty}(\rn)}\int_0^se^{-\kappa(s-\sigma)}(s-\sigma)^{-\frac{1}{2}}\sigma^{-\frac{1}{2}}\|(u\textbf{1}_{B(x,2\sqrt{t})})(\sigma)\|_{L^2(\rn)}d\sigma\\
			&\leq C\|v\|_{\X_{1,T}}\int_0^s(s-\sigma)^{-\frac{1}{2}}\sigma^{-\frac{1}{2}}\|(u\textbf{1}_{B(x,2\sqrt{t})})(\sigma)\|_{L^2(\rn)}d\sigma
		\end{align*}
		and thus by Lemma \ref{lem:bound-E} with $p=2$, $\alpha=\beta=1/2$, it holds that
		\begin{align}\label{B51}
			\nonumber[B_{41}(u,v)]_{Car}&=\sup_{x\in \rn,0<t\leq T}|B(x,\sqrt{t})|^{-1}\int_0^{t}\int_{B(x,\sqrt{t})}|B_{51}(u,v)(y,s)|^2dyds\\
			\nonumber&\leq C\sup_{x\in \rn,0<t\leq T}|B(x,\sqrt{t})|^{-1}\int_0^{t}\|B_{52}(u,v)(\cdot,s)\|_{L^2(\rn)}^2ds\\
			\nonumber&\leq C\|v\|_{\X_{1,T}}\sup_{x\in \rn,0<t\leq T}|B(x,\sqrt{t})|^{-1}\int_0^{t}\|\textbf{1}_{B(x,2\sqrt{t})}u(\cdot,s)\|^2_{L^2(\rn)}ds\\
			&\leq C\|u\|_{\X_{3,T}}\|v\|_{\X_{1,T}}.
		\end{align}
		On the other hand,
		\begin{align*}
			|B_{42}(u,v)(y,s)|
			&= \int^s_0\int_{\rn\setminus B(x,2\sqrt{t})}e^{-\kappa(t-s)}|k_1(y-z,s-\sigma)|\big|[u\cdot\nabla v](z,\sigma)\big|dzd\sigma\\
			&\leq C\sum_{j=1}^{\infty}\int^s_0\int_{A_j(x)}\dfrac{e^{-\kappa(s-\sigma)}(u\cdot \nabla v)(z,\sigma)dzd\sigma}{(|y-z|+(s-\sigma)^{1/2})^{N+1}}\\
			&\leq C\sum_{j=1}^{\infty}\int^s_0\int_{A_j(x)}|y-z|^{-(N+1)}e^{-\kappa(s-\sigma)}(u\cdot \nabla v)(z,\sigma)dzd\sigma\\
			&\leq C\sum_{j=1}^{\infty}j^{-(N+1)}\sum_{z\in A_j(x)\cap \sqrt{t}\mathbb{Z}^N}t^{\frac{-(N+1)}{2}}\bigg(\int_0^{t}\int_{B_{\sqrt{t}}(z)}|(u\cdot \nabla v)(y,s)|dyds\bigg)^{\frac{1}{2}}\\
			&\leq C\bigg(\sum_{j=1}^{\infty}j^{-2}\bigg)t^{-1/2}\|u\|_{\X_{3,T}}\|v\|_{\X_{1,T}}
		\end{align*}
		where the last estimate follows from H\"older's inequality. This implies that 
		\[[B_{42}(u,v)]_{Car}\leq C\|u\|_{\X_{3,T}}\|v\|_{\X_{1,T}}.\]
		Combining this with \eqref{B51}, we obtain \eqref{Carl-B5}. The proof of Lemma \ref{lem:DCNS} is now complete.
	\end{proof}
	
	\section{Proofs of main results}\label{sec:4}
	We mainly present the proofs of the local well-posedness results since the existence of global-in-time solutions is a direct consequence of Lemmas \ref{lem:smooth-effect}, \ref{lem:bilinear-est} and the contraction mapping principle. Also, the uniqueness criterion in Theorem \ref{thm:uniqueness} and \ref{thm:GWP2} may be established via similar arguments -- we therefore only present the proof of the former.  
	\subsection*{Proof of Theorem \ref{thm:LWP}}Let $d_0\in UC(\rn)$ and $\varGamma_{\delta_0}$ as in Lemma \ref{lem:UC-pert}. Make the ansatz $\overline{c}=c-\varGamma_{\delta_0}$ and observe that $\overline{c}$ solves the Cauchy problem 
	\begin{align*}
		\partial_t\overline{c}-\Delta\overline{c}&=-(\overline{c}-\varGamma_{\delta_0})n-u\cdot(\nabla \overline{c}+\nabla \varGamma_{\delta_0})-\Delta \varGamma_{\delta_0}\hspace{0.2cm}\mbox{in} \hspace{0.2cm}\rn\times \mathbb{R}^+\\
		\overline{c}(0)&=c_0-\varGamma_{\delta_0} \hspace{0.2cm}\mbox{on} \hspace{0.2cm}\rn.
	\end{align*}
	It is clear that the conclusion of Theorem \ref{thm:LWP} now follows from the well-posedness of the following system of equations 
	\begin{align}\label{Duhamel-CNS1}
		\begin{cases}
			\displaystyle\overline{c}=e^{t\Delta}\overline{c}_0-\int_0^te^{(t-s)\Delta}[(\overline{c}+\varGamma_{\delta_0}) n+u\cdot\nabla (\overline{c}+\nabla\varGamma_{\delta_0})+\Delta\varGamma_{\delta_0}](\cdot,s)ds\\
			\displaystyle n=e^{t\Delta}n_0-\int_0^te^{(t-s)\Delta}\nabla\cdot[n\nabla (\overline{c}+\varGamma_{\delta_0})+nu](\cdot,s)ds\\
			\displaystyle u=e^{t\Delta}u_0-\int_0^te^{(t-s)\Delta}\textbf{P}\nabla\cdot (u\otimes u)(\cdot,s)ds-\int_0^te^{(t-s)\Delta} \textbf{P}(n\nabla \Phi)(\cdot,s)ds.
		\end{cases}    
	\end{align}
	Next, define the maps 
	\begin{align*}
		\mathbb{F}_1(\overline{c},n,u,\varGamma_{\delta_0})&=B_1(\overline{c}+\varGamma_{\delta_0},n)+B_1(u,\nabla(\overline{c}+\varGamma_{\delta_0}))+\int_{0}^te^{(t-s)\Delta}\Delta\varGamma_{\delta_0}ds\\
		\mathbb{F}_2(\overline{c},n,u,\varGamma_{\delta_0})&=B_2(n,\nabla(\overline{c}+\varGamma_{\delta_0}))+B_2(n,u)\\
		\mathbb{F}_3(n,u,\nabla\Phi)&=B_3(u,u)+\mathscr{L}_{\Phi}(n)
	\end{align*}
	where $B_j$, $j=1,2,3$ are as in Section \ref{sec:3}. Combining Lemmas \ref{lem:UC-pert} and \ref{lem:bilinear-est} we obtain the following result.
	\begin{proposition}\label{prop:self-contraction}
		Let $\varGamma_{\delta_0}=e^{\delta^2_0\Delta}d_0$. Let $R>0$ and put $T_0=\min(\delta_0,R)$. Given $[\overline{c},n,u]$ in $\X_{T_0^2}$ and $\nabla\Phi\in M_{2,N-2}$, one has 
		\begin{align}
			\|\mathbb{F}_1(\overline{c},n,u,\varGamma_{\delta_0})\|_{\X_{1,T^2_0}}&\leq C_1(\|\overline{c}\|_{\X_{1,T_0^2}}+1)\|n\|_{\X_{2,T^2_0}}+C_1\|u\|_{\X_{3,T^2_0}}(\|\overline{c}\|_{\X_{1,T^2_0}}+\varepsilon_0)+C_1\varepsilon_0\\
			\|\mathbb{F}_2(\overline{c},n,u,\varGamma_{\delta_0})\|_{\X_{2,T^2_0}}&\leq C_2(\|\overline{c}\|_{\X_{1,T_0^2}}+\|u\|_{\X_{3,T^2_0}}+\varepsilon_0)\|n\|_{\X_{2,T^2_0}}\\
			\|\mathbb{F}_3(n,u,\nabla\Phi)\|_{\X_{3,T^2_0}}&\leq C_3(\|u\|^2_{\X_{3,T_0^2}}+\|n\|_{\X_{2,T^2_0}}\|\nabla\Phi\|_{M_{2,N-2}(\rn)})
		\end{align}
		for some constants $C_1,C_2,C_3>0$.
		Moreover, for any $[\overline{c}_1,n_1,u_1]\in \X_{T^2_0}$, it holds that 
		\begin{align*}
			\big\|\mathbb{F}_1(\overline{c},n,u,\varGamma_{\delta_0})-\mathbb{F}_1(\overline{c}_1,n_1,u_1,\varGamma_{\delta_0})\big\|_{\X_{1,T^2_0}}&\leq C_1(\|\overline{c}\|^2_{\X_{1,T_0^2}}+1)\|n-n_1\|_{\X_{2,T_0^2}}+\\
			C_1\|n_1\|_{\X_{2,T_0^2}}&(\|\overline{c}-\overline{c}_1\|^2_{\X_{1,T_0^2}})+C_1(\|\overline{c}\|_{\X_{1,T_0^2}}+\varepsilon_0)\|u-u_1\|_{\X_{3,T_0^2}}\\
			\big\|\mathbb{F}_2(\overline{c},n,u,\varGamma_{\delta_0})-\mathbb{F}_2(\overline{c}_1,n_1,u_1,\varGamma_{\delta_0})\big\|_{\X_{2,T^2_0}}&\leq C_2(\varepsilon_0\|\overline{c}_1\|_{\X_{1,T_0^2}}+\|u\|_{\X_{3,T_0^2}})\|n-n_1\|_{\X_{2,T_0^2}}+\\
			&C_2(\|n_1\|_{\X_{2,T_0^2}}\|u-u_1\|_{\X_{3,T_0^2}}+\|n\|_{\X_{2,T_0^2}}\|\overline{c}-\overline{c}_1\|_{\X_{1,T_0^2}})\\
			\big\|\mathbb{F}_3(n,u,\nabla\Phi)-\mathbb{F}_3(n_1,u_1,\nabla\Phi)\big\|_{\X_{3,T^2_0}}&\leq C_3(\|u\|_{\X_{3,T_0^2}}+\|u_1\|_{\X_{3,T_0^2}})\|u-u_1\|_{\X_{3,T_0^2}}+\\
			&\hspace{3.cm}C_3\|n-n_1\|_{\X_{2,T_0^2}}\|\nabla \Phi\|_{M_{2,N-2}(\rn)}.
		\end{align*}
	\end{proposition}
	From Lemma \ref{lem:UC-pert}, we remark that $\|\overline{c}_0\|_{L^{\infty}(\rn)}\leq \|c_0-d_0\|_{L^{\infty}(\rn)}+C\varepsilon_0$ for some $C>0$. 
	Let $R>0$ fixed and $\overline{\varepsilon}_0>\varepsilon_0$. Assume that \[\|c_0-d_0\|_{L^{\infty}(\rn)}+\|n_0\|_{\mathscr{L}^{-1}_{2,N-2;R}(\rn)}+\|u_0\|_{BMO^{-1}_R(\rn)}<\overline{\varepsilon}_0,\]
	then by Lemma \ref{lem:smooth-effect}, it holds that 
	\begin{align}\label{eq:lin-est}
		\big\|[e^{t\Delta}\overline{c}_0,e^{t\Delta}n_0,e^{t\Delta}u_0]\big\|_{\X_{T^2_0}}\leq C_0\overline{\varepsilon}_0.    
	\end{align}
	Let $U=[\overline{c},n,u]$ and introduce the map
	\begin{align*}
		\mathbb{F}(U)=(e^{t\Delta}\overline{c}_0+\mathbb{F}_1(U,\varGamma_{\delta_0}),e^{t\Delta}n_0+\mathbb{F}_2(U,\varGamma_{\delta_0}),e^{t\Delta}u_0+\mathbb{F}_3(n,u,\nabla\Phi)).    
	\end{align*}
	Using Proposition \ref{prop:self-contraction} together with \eqref{eq:lin-est}, we have 
	\begin{align*}
		\|\mathbb{F}(U)\|_{\X_{T^2_0}}&\leq \|e^{t\Delta}\overline{c}_0,e^{t\Delta}n_0,e^{t\Delta}u_0)\|_{\X_{T^2_0}}+\|\mathbb{F}_1(U,\varGamma_{\delta_0})\|_{\X_{1,T^2_0}}+\|\mathbb{F}_2(U,\varGamma_{\delta_0})\|_{\X_{2,T^2_0}}+\\
		&\hspace{6cm}\big\|\mathbb{F}_3(e^{t\Delta}n_0+\mathbb{F}_2(U,\varGamma_{\delta_0}),u,\nabla \Phi)\big\|_{\X_{3,T^2_0}}\\
		&\leq (2C_0+C_1)\overline{\varepsilon}_0+((C_1+C_2+C_1C_2)\varepsilon_0+\varepsilon_0^2C_2C_3)\|U\|_{\X_{T^2_0}}+\\
		&\hspace{5cm}(2C_1+2C_2+C_3+2\varepsilon_0C_2(C_3+C_1))\|U\|^2_{\X_{T^2_0}}\\
		&\leq \overline{C}\overline{\varepsilon}_0
	\end{align*}
	for some $\overline{C}>0$ provided $\|U\|_{\X_{T^2_0}}\leq \overline{C}\overline{\varepsilon}_0$ and $\varepsilon_0$ is chosen sufficiently small.
	On the other hand, we similarly show that $\mathbb{F}$ is a contraction on $B_{\overline{C}\overline{\varepsilon}_0}=\{U\in \X_{T^2_0}:\|U\|_{\X_{T^2_0}}\leq \overline{C}\overline{\varepsilon}_0\}$. This implies that $\mathbb{F}$ has a unique fixed point in $B_{\overline{C}\overline{\varepsilon}_0}$ and concludes the proof of Theorem \ref{thm:LWP}.
	\subsection*{Proof of Proposition \ref{prop:mass-con}}
	We prove here that solutions constructed satisfy the expected qualitative properties. Starting with the mass preservation, assume that $n_0\in L^1(\rn)\cap \overline{V\mathscr{L}_{2,N-2}^{-1}}(\rn)$. We look for a local solution in the Banach space $\mathbb{Y}=C([0,T_{0}),L^1(\rn))\cap \X_{2,T_{0}^2}$  equipped with the norm 
	\begin{equation*}
		\|n\|_{\mathbb{Y}}=\sup_{0\leq t<T_{0}}\|n(t)\|_{L^1(\rn)}+\|n\|_{\X_{2,T^2_{0}}},    
	\end{equation*} as a fixed point of the operator 
	\begin{equation*}
		\mathscr{P}n(x,t)=(Sn_0)(x,t)-B_2(n,u+\nabla c)(x,t).    
	\end{equation*}
	Set $K=\|n_0\|_{L^1(\rn)}$ and for $r\geq 4K$, define the closed ball $B_r=\{n\in \mathbb{Y}:\|n\|_{\mathbb{Y}}\leq r\}$.
	By the heat smoothing effect, it holds that $\|(Sn_0)(\cdot,t)\|_{L^1(\rn)}=\|e^{t\Delta}n_0\|_{L^1(\rn)}\leq K$ and for $0<t<T_0$, we have
	\begin{align*}
		\|B_2(n,u+\nabla c)(t)\|_{L^1(\rn)}&\leq C\int_{0}^t(t-s)^{-1/N}\|n(u+\nabla c)(s)\|_{L^1(\rn)}\\
		&\leq C\int_{0}^t(t-s)^{-1/N}\|n(s)\|_{L^1(\rn)}\|(u+\nabla c)(s)\|_{L^{\infty}(\rn)}ds\\
		&\leq C\|n\|_{\mathbb{Y}}(\|u\|_{\X_{3,T^2_{0}}}+\|c\|_{\X_{1,T^2_{0}}})\int_{0}^t(t-s)^{-1/N}s^{-1/2}ds\\
		&\leq CT_{0}^{-1/N+1/2}\|n\|_{\mathbb{Y}}(\|u\|_{\X_{3,T^2_{0}}}+\|c\|_{\X_{1,T^2_{0}}})\int_{0}^1(1-s)^{-1/N}s^{-1/2}ds\\
		&\leq C'T_{0}^{-1/N+1/2}\|n\|_{\mathbb{Y}}.
	\end{align*}
	All together, we have 
	\begin{align*}
		\|[\mathscr{P}n](t)\|_{L^1(\rn)}\leq K+C'T_{0}^{-1/N+1/2}r  
	\end{align*}
	for all $n\in B_r$ and thus $\mathscr{P}(B_r)\subset B_r$ provided $T_0$ is chosen small.  
	Also, if $m\in \mathbb{Y}$ then the same calculations lead to 
	\begin{align*}
		\|[\mathscr{P}n](t)-[\mathscr{P}m](t)\|_{L^1(\rn)}\leq C'T_{0}^{-1/N+1/2}\|n-m\|_{\mathbb{Y}},\quad \mbox{for all} \hspace{0.2cm}n,m\in B_r.    
	\end{align*}
	If $T_{0}$ is chosen small, the conclusion follows from the contraction principle, that is, with the result of the previous subsection, $\mathscr{P}$ has a fixed point in $\mathbb{Y}$. That $n\in C([0,T_0),L^1(\rn))$ goes by a standard continuity argument, thus details are omitted. Now, by Fubini's theorem one may write for $0<s<t<T_0$,
	\begin{align}\label{mass-conserv}
		\nonumber\int_{\rn}n(x,t)dx&=\int_{\rn}e^{t\Delta}n_0(x)dx-\int_{\rn}\int_0^te^{(t-s)\Delta}\nabla\cdot [n(u+\nabla c)](s)dsdx\\
		&=\int_{\rn}n_0(x)dx-\int_0^t\int_{\rn}\nabla e^{(t-s)\Delta}\cdot [n(u+\nabla c)](s)dxds.
	\end{align}
	In view of the divergence theorem and the density of the space of smooth, compactly supported functions in $L^1(\rn)$ it follows that
	if $H$ together with $\nabla\cdot H$ belong to $L^1(\rn)$, then it is true that $\int_{\rn}\nabla\cdot H dx=0$.
	If we set $H(x)=e^{(t-s)\Delta} (n(u+\nabla c))(x)$, then from the above calculation, it follows that $H\in L^1(\rn)^N$ and $\nabla\cdot H \in L^1(\rn)$. In fact,
	\begin{equation*}
		\|H\|_{L^1(\rn)}\leq Ct^{-1/2}\|n(t)\|_{L^1(\rn)}(\|c\|_{\X_{1,T^2_0}}+\|u\|_{\X_{3,T^2_0}})    
	\end{equation*}
	and 
	\begin{equation*}
		\|\nabla\cdot H\|_{L^1(\rn)}\leq C(t-s)^{-1/N}s^{-1/2}\|n(s)\|_{L^1(\rn)}(\|c\|_{\X_{1,T^2_0}}+\|u\|_{\X_{3,T^2_0}}),\quad 0<s<t
	\end{equation*}
	so that by applying this in \eqref{mass-conserv}, one finds that \begin{equation}\label{mass}
		\int_{\rn}n(x,t)dx=\int_{\rn}n_0(x)dx    
	\end{equation} for each $t\in (0,T_0)$. For the rest of the proof, slightly more assumptions on $u_0$ and $c_0$ are needed for technical reasons (e.g. it is used to show the contraction property pertaining to \eqref{main-eq}$_2$ which may be obtained from the mild formulation). It is enough that $u_0\in L^{\infty}(\rn)$ and $\nabla c_0\in L^{\infty}(\rn)$. As a result, the constructed local solution is such that $\nabla c$ and $u$ are bounded in space and time. The former statement is easy to check and the latter is known \cite{C}.  Since $n\in C([0,T_0);L^1(\rn))$, it can be shown using classical regularity arguments that $n$  enjoys a better  regularity, $n\in C((0,T_0);W^{2,p}(\rn))\cap C^1((0,T_0);L^{p}(\rn))$ for every $p\in (1,\infty)$. Hence,  multiplying the equation for $n$, \eqref{main-eq}$_2$ by $\mathrm{sgn}\hspace{0.1cm} n$ and integrating over $\rn$, we obtain 
	\begin{equation*}
		\dfrac{d}{dt}\int_{\rn}|n(x,t)|dx\leq 0\hspace{0.2cm}\mbox{for all} \hspace{0.2cm}t\in (0,T_0).
	\end{equation*}
	In turn, the latter implies the following property   
	\begin{equation}\label{L1-contr}
		\|n(t)\|_{L^1(\rn)}\leq \|n_0\|_{L^1(\rn)},\quad t\in (0,T_0).
	\end{equation}
	Such a contraction property partially relies on the fact that $T(t)=e^{t\Delta}$ defines a contraction semigroup on $L^1(\rn)$ and is known to hold for a large class of nonlinear equations with more general principal operators and nonlinearities with divergence structure of the kind similar to that in \eqref{main-eq}$_{2}$. We refer the reader to \cite{BFW,EZ} and references therein for more details. With \eqref{mass}-\eqref{L1-contr} and denoting by $n^{+}=\max(n,0)$ and $n^{-}=-\min(n,0)$ both integrable, we compute
	\begin{align*}
		\int_{\rn}n^{-}(x,t)dx&=-\int_{\rn}\dfrac{n(x,t)-|n(x,t)|}{2}dx\\
		&\leq \frac{1}{2}\int_{\rn}|n_0(x)|dx-\frac{1}{2}\int_{\rn}n_0(x)dx\\
		&\leq \int_{\rn}n^{-}_0(x)dx=0
	\end{align*}
	because $n_0\geq 0$. The conclusion $n(x,t)\geq 0$ for almost every $(x,t)\in \rn\times (0,T_0)$ follows from the non-negativity of $n^{-}$.
	
	In order to prove the sign preservation for $c$  we  argue at the level of smooth solutions (obtained via an approximation procedure) as the non-negativity property of those is guaranteed as soon as the regularized data is. The next step then consists of obtaining good uniform estimate for the difference of solutions in some suitable topology and then pass to the limit so that the solution $c$ inherits the non-negativity feature of the approximation. As already discussed, the strong $L^{\infty}$-topology seems more suited to achieve this goal. Let $c_0\geq 0$ be a bounded uniformly continuous function. Then for each $\varepsilon>0$, $c_{0\varepsilon}=e^{\varepsilon\Delta}c_0$ is nonnegative, smooth and uniformly bounded and the solution $c_{\varepsilon}$ of the regularized problem
	\begin{equation}
		\partial_tc_{\varepsilon}-\Delta c_{\varepsilon}+u\cdot \nabla c^{\varepsilon}+nc_{\varepsilon}=0\hspace{0.2cm}\mbox{in}\hspace{0.2cm}\rn\times (0,T_0)\quad c_{\varepsilon}(\cdot,0)=c_{0\varepsilon}    
	\end{equation}
	is also bounded smooth and nonnegative by the maximum principle. Observe that the above Sobolev regularity for $n$ and standard embedding theorem imply $n(t)\in L^{q}(\rn)$ for $q>N/2$ and $A=\displaystyle\sup_{0<t\leq T_0}\|n(t)\|_{L^q(\rn)}<\infty$. Hence,
	\begin{align*}
		\|(c_{\varepsilon}-c)(t)&\|_{L^{\infty}(\rn)}\\
		&\leq \|e^{t\Delta}c_{0\varepsilon}-e^{t\Delta}c_0\|_{L^{\infty}(\rn)}+ \int_0^t\big\|\nabla  e^{(t-s)\Delta}\cdot[u(c-c_{\varepsilon})](s)\big\|_{L^{\infty}(\rn)}ds+\\
		&\hspace{7cm}\int_0^t\big\| e^{(t-s)\Delta}[n(c-c_{\varepsilon})](s)\big\|_{L^{\infty}(\rn)}ds\\
		&\leq \|e^{t\Delta}c_{0\varepsilon}-e^{t\Delta}c_0\|_{L^{\infty}(\rn)}+\|u\|_{\X_{3,T^2_0}}\int_{0}^t(t-s)^{-1/N}s^{-1/2}\|(c_{\varepsilon}-c)(s)\|_{L^{\infty}(\rn)}ds+\\
		&\hspace{7cm}A\int_{0}^t(t-s)^{-N/2q}\|(c_{\varepsilon}-c)(s)\|_{L^{\infty}(\rn)}ds.
	\end{align*}
	Since the kernel $s\mapsto H(t,s)=(t-s)^{-1/N}s^{-1/2}+(t-s)^{-N/2q}\in L^1([0,t])$, we can apply a Gronwall's type inequality (see for instance \cite[Page 58]{W}) to arrive at
	\begin{align}\label{gronwall}
		\|(c_{\varepsilon}-c)(t)&\|_{L^{\infty}(\rn)}\leq \|e^{t\Delta}c_{0\varepsilon}-e^{t\Delta}c_0\|_{L^{\infty}(\rn)}+ C\int_0^tH^{\star}(t,s)\big\|(e^{t\Delta}c_{0\varepsilon}-e^{t\Delta}c_0)(s)\big\|_{L^{\infty}(\rn)}ds
	\end{align}
	where $s\mapsto H^{\star}(t,s)$ is integrable on $[0,t]$ and $C=\max(A,\|u\|_{\X_{3,T^2_0}})$. Using the semigroup property, it follows that 
	\[\lim_{\varepsilon\rightarrow 0}\|e^{t\Delta}c_{0\varepsilon}-e^{t\Delta}c_0\|_{L^{\infty}(\rn)}=\lim_{\varepsilon\rightarrow 0}\|e^{\varepsilon\Delta}\widetilde{c_0}-\widetilde{c_0}\|_{L^{\infty}(\rn)}=0\]
	because $\widetilde{c_0}=e^{t\Delta}c_0$ is bounded uniformly continuous. Passing to the limit in \eqref{gronwall} yields the pointwise almost everywhere convergence of $c_{\varepsilon}$ to $c$. Hence, $c\geq 0$ a.e. in $\rn\times [0,T_0)$.
	\subsection*{Proof of Theorem \ref{thm:uniqueness}}
	Assume that $\Phi\in \mathcal{S}'(\rn)$ such that $\nabla \Phi\in M_{2,N-2}(\rn)$ and let $U_j=[c_j,n_j,u_j]\in L^{\infty}_{loc}((0,\infty);L^{\infty}(\rn))$, $j=1,2$ and set $U=U_1-U_2$. Through similar ideas as in  \cite{Miu}, we first show that $U=0$ on $\rn\times [0,T_0)$ for some $T_0>0$. By Lemma \ref{lem:bilinear-est} and the identities
	\begin{align*}
		(c_1-c_2)&=B_1(c_2,n_2-n_1))+B_1(c_2-c_1,n_1))+B_1(u_2,\nabla (c_2-c_1)))+B_1(u_2-u_1,\nabla c_1)\\
		(n_1-n_2)&=B_2(n_1-n_2,\nabla c_2)+B_2(n_1,\nabla (c_2-c_1))+B_2(n_2,u_2-u_1)+B_2(u_1,n_2-n_1)\\
		(u_1-u_2)&=B_3(u_2,u_2-u_1)+B_3(u_2-u_1,u_1)+\mathscr{L}_{\Phi}(n_2-n_1),
	\end{align*}
	it follows that 
	\begin{align}\label{bound-on-U}
		\nonumber\|U\|_{\X_T}&=\|c-c_1\|_{\X_{1,T}}+\|n_1-n_2\|_{\X_{2,T}}+\|u_1-u_2\|_{\X_{2,T}}\\
		&\leq K_1\big(\|U_1\|_{\X_T}+\|U_2\|_{\X_T}+\|\nabla\Phi\|_{M_{2,N-2}(\rn)}\big)\|U\|_{\X_T}.
	\end{align}
	In view of condition \eqref{limit-cond}, there exists $T_0>0$ such that $\|U_1\|_{\X_{T_0}}+\|U_2\|_{\X_{T_0}}\leq \dfrac{1}{4K_1}$. Given $\varepsilon\in (0,1)$, if $\|\nabla\Phi\|_{M_{2,N-2}(\rn)}\leq \dfrac{\varepsilon}{4K_1}$, then \eqref{bound-on-U} implies that 
	\begin{align*}
		\|U\|_{\X_{T_0}}\leq \dfrac{1}{2}\|U\|_{\X_{T_0}}.    
	\end{align*}
	Hence, $U_1=U_2$ on $\rn\times [0,T_0)$. To extend this property to the whole interval $[0,\infty)$, observe that \[K_{12}\equiv\displaystyle\sup_{s\in (T_0,T)}\|U_1(s)\|_{L^{\infty}(\rn)}+\sup_{s\in (T_0,T)}\|U_2(s)\|_{L^{\infty}(\rn)}<\infty\] since $U_1,U_2\in L^{\infty}_{loc}(0,\infty;L^{\infty}(\rn))$. Now set
	\begin{align*}
		a(t)=\sup_{T_0<s<t}\|U(s)\|_{L^{\infty}(\rn)}    
	\end{align*}
	for $t>T_0$. We claim that there exits $\tau:=\tau(T_0)$ such that $U_1=U_2$ on $\rn\times [0,T_0+\tau)$. To see this, compute
	\begin{align*}
		|(c_1-c_2)(x,t)|&=\bigg|\int_0^t\int_{\rn}g(x-y,t-s)[c_2(n_2-n_1)+n_1(c_2-c_1)+\\
		&\hspace{7cm}u_2\nabla(c_2-c_1)+(u_2-u_1)\nabla c_1]dyds\bigg|\\
		&\leq \int_{T_0}^t\int_{\rn}g(x-y,t-s)(|c_2(n_2-n_1)|+|n_1(c_2-c_1)|+\\
		&\hspace{7cm}|u_2\nabla(c_2-c_1)|+|(u_2-u_1)\nabla c_1|)dyds\\
		&\leq C_1K_{12}a(t)\int_{T_0}^t\int_{\rn}g(x-y,t-s)dyds\\
		&\leq C_1K_{12}a(t)(t-T_0).
	\end{align*}
	On the other hand,
	\begin{align*}
		|(n_1-n_2)(x,t)|&=\bigg|\int_0^t\int_{\rn}g(x-y,t-s)\nabla\cdot[(n_2-n_1)\nabla c_2+n_1\nabla(c_2-c_1)+\\
		&\hspace{7cm}n_2(u_2-u_1)+u_1(n_2-n_1)]dyds\bigg|\\
		&\leq \int_{T_0}^t\int_{\rn}|\nabla g(x-y,t-s)|\big(|(n_2-n_1)\nabla c_2|+|n_1\nabla(c_2-c_1)|+\\
		&\hspace{7cm}|n_2(u_2-u_1)|+|u_1(n_2-n_1)|\big)dyds\\
		&\leq C_2K_{12}a(t)\int_{T_0}^t\int_{\rn}\dfrac{dy}{[(x-y)^2+(t-s)]^{\frac{N+1}{2}}}ds\\
		&\leq C_2K_{12}a(t)\int_{T_0}^t(t-s)^{-1/2}ds\int_{\rn}\dfrac{dy}{(|y|^2+1)^{\frac{N+1}{2}}}\\
		&\leq C_2K_{12}a(t)\sqrt{t-T_0}.
	\end{align*}
	Also, using \eqref{Oseen}, we find that
	\begin{align*}
		|(u_1-u_2)(x,t)|&=\bigg|\int_0^t\int_{\rn}\nabla k_2(x-y,t-s)[u_2\otimes(u_2-u_1)+(u_2-u_1)\otimes u_1]dyds+\\
		&\hspace{5cm}\int_0^t\int_{\rn} k_2(x-y,t-s)(n_2-n_1)\nabla\Phi]dyds\bigg|\\
		&\leq \int_{T_0}^t\int_{\rn}|\nabla k_2(x-y,t-s)||u_2\otimes(u_2-u_1)|+|(u_2-u_1)\otimes u_1|dyds+\\
		&\hspace{5cm}\int_{T_0}^t\int_{\rn}|k_2(x-y,t-s)||(n_2-n_1)\nabla\Phi|dyds\\
		&\leq C_3K_{12}a(t)\int_{T_0}^t\int_{\rn}\dfrac{dy}{[(x-y)^2+(t-s)]^{\frac{N+1}{2}}}ds+\\
		&\hspace{6cm}C_4\int_{T_0}^t\int_{\rn}\dfrac{|(n_2-n_1)(s)\nabla \Phi|dy}{(|x-y|+(t-s)^{1/2})^{-N}}ds\\
		&\leq C_3K_{12}a(t)\sqrt{t-T_0}+ C_4\int_{T_0}^t(t-s)^{-1/2}\|(n_2-n_1)(s)\nabla \Phi\|_{M_{2,N-2}(\rn)}ds\\
		&\leq C_3K_{12}a(t)\sqrt{t-T_0}+ C_4a(t)\|\nabla \Phi\|_{M_{2,N-2}(\rn)}\int_{T_0}^t(t-s)^{-1/2}ds\\
		&\leq C_3K_{12}a(t)\sqrt{t-T_0}+ C_4a(t)\|\nabla \Phi\|_{M_{2,N-2}(\rn)}\sqrt{t-T_0}.
	\end{align*}
	Summarizing, we have that (putting $C_5=\max(C_1,C_2+C_3,C_4\|\nabla \Phi\|_{M_{2,N-2}(\rn)})$)
	\begin{align*}
		|U(x,t)|\leq C_5K_{12}a(t)\bigg((t-T_0)+\frac{K_{12}+1}{K_{12}}(t-T_0)^{\frac{1}{2}}\bigg).     
	\end{align*}
	This implies that $a(T_0+\tau)\leq \frac{1}{4}a(T_0+\tau) $
	for $\tau=\theta^2$, $\theta=\sqrt{\bigg(\dfrac{K_{12}+1}{K_{12}}\bigg)^2+\dfrac{1}{C_5K_{12}}}-\dfrac{K_{12}+1}{K_{12}}>0$. This shows the claim  and a simple iteration procedure yields the desired conclusion.
	\subsection*{Proof of Theorem \ref{thm:LWP1}}The argument here is similar to that used above. We give the details for the reader's convenience. Let $d_0\in UC(\rn)$, $\varGamma_{\delta_0}$ and $\overline{c}=c-\varGamma_{\delta_0}$ as before. Next let $\widetilde{v}_{\kappa}=e^{-t\kappa}e^{t\Delta}v_0$ and make the change of variable $w=v-\widetilde{v}_{\kappa}$. Then Syst. \eqref{Duhamel-D-CNS} becomes   
	\begin{align}\label{Duhamel-DCNS1}
		\begin{cases}
			\displaystyle\overline{c}=e^{t\Delta}\overline{c}_0-\int_0^te^{(t-s)\Delta}[(\overline{c}+\varGamma_{\delta_0}) n+u\cdot\nabla (\overline{c}+\nabla\varGamma_{\delta_0})+\Delta\varGamma_{\delta_0}](\cdot,s)ds\\
			\displaystyle n=e^{t\Delta}n_0-\int_0^te^{(t-s)\Delta}\nabla\cdot[nu+n\nabla (\overline{c}+\varGamma_{\delta_0})+n\nabla(w+\widetilde{v}_{\kappa})](\cdot,s)ds\\
			\displaystyle w=\int_0^te^{-(t-s)\kappa}e^{(t-s)\Delta} [u\cdot\nabla (w+\widetilde{v}_{\kappa})+n](\cdot,s)ds\\
			\displaystyle u=e^{-t\kappa}e^{t\Delta}u_0-\int_0^te^{-t\kappa}e^{(t-s)\Delta}\textbf{P}\nabla\cdot (u\otimes u)(\cdot,s)ds-\int_0^te^{(t-s)\Delta} \textbf{P}(n \Psi)(\cdot,s)ds.
		\end{cases}    
	\end{align}
	Observe that if $v_0\in BMO_R(\rn)$ for some $0<R\leq \infty$, then 
	\begin{align*}
		[\widetilde{v}_{\kappa}]_{\X_{1,R^2}}\leq C\|v_0\|_{BMO_R(\rn)}.    
	\end{align*}
	This easily follows from  the Carleson measure characterization of $BMO_R(\rn)$. Remark that the  kernel of the integral operator $e^{-t\kappa}e^{t\Delta}$, $\kappa>0$ is bounded above by the heat kernel.  This in turn implies that $\nabla \widetilde{v}_{\kappa}\in \X_{3,R^2}$. Hence, setting \begin{align*}
		\mathbb{F}_{2,\kappa}(\overline{c},n,w,u)&=B_2(n,\nabla(\overline{c}+\varGamma_{\delta_0}))+B_2(n,u)+B_2(n,\nabla(w+\widetilde{v}_{\kappa}))\\
		\mathbb{F}_4(n,w,u)&=B_4(u,w+\widetilde{v}_{\kappa})+\mathscr{L}(n),
	\end{align*}
	the next Proposition follows from Lemmas \ref{lem:bilinear-est} and \ref{lem:DCNS}.
	\begin{proposition}\label{prop:self-cont}
		For $R>0$ and $T_0=\min(\delta_0,R)$. If $[n,w,u]\in \X_{2,T_0^2}\times \X_{1,T_0^2}\times \X_{3,T_0^2}$  then $\mathbb{F}_{2,\kappa}(\overline{c},n,w,u)\in \X_{2,T_0^2}$,  $\mathbb{F}_4(n,w,u)\in \X_{1,T_0^2}$ and there exists $C_4,C_5>0$ such that
		\begin{align}
			\|\mathbb{F}_{2,\kappa}(\overline{c},n,w,u)\|_{\X_{2,T^2_0}}&\leq C_5(\|\overline{c}\|_{\X_{1,T_0^2}}+\|u\|_{\X_{3,T^2_0}}+[\widetilde{v}_{\kappa}]_{\X_{1,T_0^2}}+\|w\|_{\X_{1,T_0^2}}+\varepsilon_0)\|n\|_{\X_{2,T^2_0}}\\
			\|\mathbb{F}_4(n,w,u)\|_{\X_{1,T^2_0}}&\leq C_4([\widetilde{v}_{\kappa}]_{\X_{1,T_0^2}}+\|w\|_{\X_{1,T^2_0}})\|u\|_{\X_{3,T^2_0}}+C_4\|n\|_{\X_{2,T^2_0}}.
		\end{align}
		In addition, for any $[\overline{c}_1,n_1,w_1,u_1]\in \Z_{T^2_0}$, it holds that  
		\begin{align*}
			&\big\|\mathbb{F}_{2,\kappa}(\overline{c},n,w,u)-\mathbb{F}_{2,\kappa}(\overline{c}_1,n_1,w_1,u_1)\big\|_{\X_{2,T^2_0}}\\
			&\leq C_5\|n\|_{\X_{2,T_0^2}}\|w-w_1\|_{\X_{1,T_0^2}}+
			C_5(\|\overline{c}\|^2_{\X_{1,T_0^2}}+[\widetilde{v}_{\kappa}]_{\X_{1,T_0^2}}+\|w\|_{\X_{1,T^2_0}}+1)\|n-n_1\|_{\X_{2,T_0^2}}+\\
			& \hspace{3cm}C_5\|n_1\|_{\X_{2,T_0^2}}(\|\overline{c}-\overline{c}_1\|^2_{\X_{1,T_0^2}})+C_5(\|\overline{c}\|_{\X_{1,T_0^2}}+\varepsilon_0)\|u-u_1\|_{\X_{3,T_0^2}}
		\end{align*}
		and 
		\begin{align*}
			\big\|\mathbb{F}_4(n,w,u)-\mathbb{F}_4(n_1,w_1,u_1)\big\|_{\X_{1,T^2_0}}&\leq C_4(\|w_1\|_{\X_{1,T_0^2}}+[w_1]_{\X_{1,T_0^2}})\|u-u_1\|_{\X_{3,T_0^2}}+\\
			&\quad \quad \quad C_4\|u\|_{\X_{1,T_0^2}}\|w-w_1\|_{\X_{1,T_0^2}}+C_4\|n-n_1\|_{\X_{2,T_0^2}}.
		\end{align*}
	\end{proposition}
	The remaining part of the proof is done exactly as before. For brevity,  details are omitted.
	
	\section{Appendix}
	This section contains  all the deferred proofs which follow as particular cases of more general results.
	\begin{definition}
		Let $N\geq 3$ and $-2<\lambda \leq 2$. We say that a tempered distribution $f$ belongs to  $\mathscr{L}^{-1}_{2,N-\lambda}(\rn)$ if $\|f\|_{\mathscr{L}^{-1}_{2,N-\lambda}(\rn)}$ is finite, 
		\begin{equation*}
			\|f\|_{\mathscr{L}^{-1}_{2,N-\lambda}(\rn)}=\sup_{x\in \rn,R>0}\bigg(|B(x,R)|^{\frac{\lambda}{N}-1}\IntC |e^{t\Delta}f(y,s)|^2dyds\bigg)^{1/2}.   
		\end{equation*}
	\end{definition}
	For $\lambda=0$, $\mathscr{L}^{-1}_{2,N-\lambda}(\rn)$ is the space $BMO^{-1}(\rn)$.
	Recall the characterization of square-Campanato spaces via caloric extension  \cite{JXY}: $f\in \mathscr{L}_{2,N-\lambda}(\rn)$, $\lambda\in (-2,2]$ if and only if its caloric extension $u=e^{t\Delta}f\in T^{2,\lambda}$ and $\|u\|_{T^{2,\lambda}}\leq C\|f\|_{\mathscr{L}_{2,N-\lambda}(\rn)}$ where $C$ is a constant independent of $f$ and
	\begin{align*}
		\|u\|_{T^{2,\lambda}}:=\sup_{x\in \rn,R>0}\bigg(|B(x,R)|^{\frac{\lambda}{N}-1}\IntC |\nabla u(y,s)|^2dyds\bigg)^{1/2}.   
	\end{align*}
	This extrinsic definition of Campanato spaces suggests that $\mathscr{L}^{-1}_{2,N-\lambda}(\rn)=\nabla\cdot (\mathscr{L}_{2,N-\lambda}(\rn))^N$. This is indeed the case as shown below. 
	\begin{lemma}\label{lem:M^{-1}-charac}
		Assume that $\lambda\in (-2,2]$ and $N>2$. A tempered distribution $f\in \mathscr{L}^{-1}_{2,N-\lambda}(\rn)$ if and only if there exists a family $(f_l)^N_{l=1}\subset \mathscr{L}^{-1}_{2,N-\lambda}(\rn)$ such that $f=\sum_{l=1}^{N}f_l$. Moreover, the following equivalence holds \begin{equation}\label{M^{-1}}
			\|f\|_{\mathscr{L}^{-1}_{2,N-\lambda}(\rn)}\approx \inf\bigg\{\sum_{l=1}^{N}\|f_l\|_{\mathscr{L}^{-1}_{2,N-\lambda}(\rn)}:f=\sum_{l=1}^{N}\partial_lf_l\bigg\}.
		\end{equation} 
	\end{lemma}
	\begin{proof}[Proof of Lemma \ref{lem:M^{-1}-charac}] 
		Let $f$ be a tempered distribution. Assume that there is $f_l\in \mathscr{L}^{-1}_{2,N-\lambda}(\rn)$, $l=1,\cdots,N$ with $f=\sum_{l=1}^Nf_l$. Using the characterization of Morrey spaces by heat extension we obtain 
		\begin{align*}
			\|f\|_{\mathscr{L}^{-1}_{2,N-\lambda}(\rn)}&\leq C\sum_{l=1}^N\sup_{x\in \rn,R>0}\bigg(|B(x,R)|^{\frac{\lambda}{N}-1}\IntC |\partial_je^{t\Delta}f_l|^2dyds\bigg)^{1/2}\\
			&\leq C\sum_{l=1}^N\|f_l\|_{\mathscr{L}^{-1}_{2,N-\lambda}(\rn)}.    
		\end{align*}
		This shows that $\nabla\cdot (\mathscr{L}_{2,N-\lambda}(\rn))^N\subset \mathscr{L}^{-1}_{2,N-\lambda}(\rn)$. The converse follows from the observation that if $f\in \mathscr{L}^{-1}_{2,N-\lambda}(\rn)$ then $f_{j,l}=\partial_j\partial_l(-\Delta)^{-1}f\in \mathscr{L}^{-1}_{2,N-\lambda}(\rn)$. Indeed, let $\varphi\in C^{\infty}_0(\rn)$ be a cut-off function with supp $\varphi\subset B(0,1)$, the Euclidean unit ball and $\int_{\rn}\varphi dx=1$. Set $\varphi_{R}(x)=R^{-N}\varphi(x/R)$ for $R>0$ and write 
		\begin{align*}
			e^{t\Delta}f_{j,l}=\partial_j\partial_l(-\Delta)^{-1}(e^{t\Delta}f)=f^1_{j,l}+f^2_{j,l}
		\end{align*}
		where $f^1_{j,l}=\varphi_R\ast \partial_j\partial_l(-\Delta)^{-1}(e^{t\Delta}f)$ and $f^2_{j,l}=f_{j,l}-f^1_{j,l}$. Noticing that $\partial_j\partial_l(-\Delta)^{-1}$ is a Fourier multiplier of order $0$, one has  
		\begin{align*}
			\|f^1_{j,l}\|_{L^{\infty}(\rn)}&\leq C\|\varphi_R\|_{\dot{B}^{1+\frac{\lambda}{2}}_{1,1}(\rn)}\|f_{j,l}\|_{\dot{B}^{-(1+\frac{\lambda}{2})}_{\infty,\infty}(\rn)}\\
			&\leq CR^{-1-\frac{\lambda}{2}}\|e^{t\Delta}f\|_{\dot{B}^{-(1+\frac{\lambda}{2})}_{\infty,\infty}(\rn)}\\
			&\leq CR^{-1-\frac{\lambda}{2}}\|f\|_{\dot{B}^{-(1+\frac{\lambda}{2})}_{\infty,\infty}(\rn)}\\
			&\leq CR^{-1-\frac{\lambda}{2}}\|f\|_{\mathscr{L}^{-1}_{2,N-\lambda}(\rn)}
		\end{align*}
		since the operator $e^{t\Delta}$ maps $\dot{B}^s_{p,q}(\rn)$ into itself (for $1\leq p,q\leq \infty$, $s\in \mathbb{R}$)  in addition to $\mathscr{L}^{-1}_{2,N-\lambda}(\rn)\subset \dot{B}^{-(1+\lambda/2)}_{\infty,\infty}(\rn)$. The proof of the latter continuous embedding is given below. Using this, it follows that 
		\begin{align*}
			\IntC |f^1_{j,l}|^2dydt\leq C |B(x,R)|^{1-\frac{\lambda}{N}}\|f\|_{\mathscr{L}^{-1}_{2,N-\lambda}(\rn)}.   
		\end{align*}
		To estimate $f^2_{j,l}$, we further decompose it into two parts writing $f^2_{j,l}=f^{21}_{j,l}+f^{22}_{j,l}$ with 
		\begin{align*}
			f^{21}_{j,l}&=\partial_j\partial_l(-\Delta)^{-1}(\zeta_{x,R}e^{t\Delta}f)-\varphi_R\ast \partial_j\partial_l(-\Delta)^{-1}(\zeta_{x,R}e^{t\Delta}f)\\
			f^{22}_{j,l}&=\partial_j\partial_l(-\Delta)^{-1}[(1-\zeta_{x,R})e^{t\Delta}f] -\varphi_R\ast \partial_j\partial_l(-\Delta)^{-1}[(1-\zeta_{x,R})e^{t\Delta}f]  
		\end{align*}
		where $\zeta_{x,R}=\zeta(R^{-1}(x-\cdot))$ and $\zeta\in C^{\infty}_0(\rn)$, supp $\zeta \subset B(0,20)$, $\zeta=1$ on $B(0,10)$. By Pancherel's identity,
		\begin{align}\label{f211}
			\nonumber\IntC |\partial_j\partial_l(-\Delta)^{-1}(\zeta_{x,R}e^{t\Delta}f)|^2dydt&\leq \int_0^{R^2}\|\partial_j\partial_l(-\Delta)^{-1}(\zeta_{x,R}e^{t\Delta}f)\|^2_{L^2(\rn)}dt\\
			\nonumber&\leq C\int_0^{R^2}\big\|\xi_{j}\xi_{l}|\xi|^{-2} \mathcal{F}\big(\zeta_{x,R}e^{t\Delta}f\big)\big\|^2_{L^2(\rn)}dt\\
			\nonumber&\leq C\int_0^{R^2}\big\|\mathcal{F}\big(\zeta_{x,R}e^{t\Delta}f)\big\|^2_{L^2(\rn)}dt\\
			&\leq C\int_0^{R^2}\big\|\zeta_{x,R}e^{t\Delta}f\big\|^2_{L^2(\rn)}dt.
		\end{align}
		On the other hand, invoking Minkowski's inequality we find that
		\begin{align}\label{f212}
			\nonumber\IntC |\varphi_R\ast\partial_j\partial_l(-\Delta)^{-1}(\zeta_{x,R}e^{t\Delta}f)|^2dydt&\leq \int_0^{R^2}\|\varphi_R\ast\partial_j\partial_l(-\Delta)^{-1}(\zeta_{x,R}e^{t\Delta}f)\|^2_{L^2(\rn)}dt\\
			\nonumber&\leq C\int_0^{R^2}\|\partial_j\partial_l(-\Delta)^{-1}(\zeta_{x,R}e^{t\Delta}f)\|^2_{L^2(\rn)}dt\\
			&\leq C\int_0^{R^2}\big\|\zeta_{x,R}e^{t\Delta}f\big\|^2_{L^2(\rn)}dt.
		\end{align}
		Thus, from \eqref{f211} and \eqref{f212}, one deduces that
		\begin{align*}
			\IntC |f^{21}_{j,l}(y,t)|^2dydt\leq C|B(x,R)|^{1-\lambda/N}.   
		\end{align*}
		In order to estimate the term $f^{22}_{j,l}$, recall the pointwise estimate (see \cite[Page 161]{L})
		\begin{align}\label{pt-est-f22}
			\big|f^{22}_{j,l}(y,t)\big|\leq C\int_{|x-z|\geq 10R}\dfrac{R}{|x-z|^{N+1}}|e^{t\Delta}f(z)|dz,\quad y\in B(x,R).    
		\end{align}
		Set $\widehat{A}_k=B(x,10R(k+1))\setminus B(x,10Rk)$ and observe that \eqref{pt-est-f22} implies 
		\begin{align*}
			\int_{B(x,R)}\big|f^{22}_{j,l}(y,t)\big|^2dy&\leq CR^{N+1}\int_{|x-z|\geq 10R}\dfrac{1}{|x-z|^{N+1}}|e^{t\Delta}f(z)|^2dz\\
			&\leq CR^{N+1}\sum^{\infty}_{k=1}\int_{\widehat{A}_k}\dfrac{1}{|x-z|^{N+1}}|e^{t\Delta}f(z)|^2dz\\
			&\leq C\sum^{\infty}_{k=1}k^{-(N+1)}\sum_{\substack{w\in \mathbb{Z}^N\\ \&\\w\in \widehat{A}_k}}\int_{B(w,R)}|e^{t\Delta}f|^2dz\\
			&\leq C\bigg(\sum^{\infty}_{k=1}k^{-2}\bigg)\int_{B(w,R)}|e^{t\Delta}f|^2dz.
		\end{align*}
		Hence, 
		\begin{align*}
			|B(x,R)|^{1/\lambda-N}\IntC\big|f^{22}_{j,l}(y,t)\big|^2dydt\leq C. 
		\end{align*}
		Next, let $f_l=-\partial_l(-\Delta)^{-1}f$. It is clear that $f_l\in \mathscr{L}_{2,N-\lambda}(\rn)$ for each $l=1,...,N$. Moreover, we verify that
		\begin{align*}
			\mathcal{F}\bigg(\sum_{l=1}^N\partial_lf_l\bigg)(\xi)= \sum_{l=1}^Ni\xi_l \mathcal{F}(f_l)(\xi)=\sum_{l=1}^N-i\xi_l(i\xi_l)|\xi|^{-2} \mathcal{F}(f)(\xi)= \mathcal{F}(f)(\xi).  
		\end{align*}
		This achieves the proof of Lemma \ref{lem:M^{-1}-charac}.
	\end{proof}
	\begin{proof}[Proof of the embedding \eqref{embedding-BM-Y_Besov}]
		Let  $N\geq 3$. Consider $0\leq \beta<N$, $0\leq \lambda\leq 2$ and take $2\leq p<\frac{2(N-\beta)}{\lambda}$ ($2\leq p<\infty$ when $\lambda=0$). Assume that $f\in \mathcal{N}^{-2s}_{p,\beta,\infty}(\rn)$, $s=\frac{\lambda+2}{4}+\frac{\beta-N}{2p}>0$. From the characterization of Besov-Morrey spaces (see e.g. \cite{KY,M}), it holds that 
		\begin{equation}
			\sup_{t>0} t^{s}\|e^{t\Delta}f\|_{M_{p,\beta}(\rn)}\approx \|f\|_{\mathcal{N}^{-2s}_{p,\beta,\infty}(\rn)}.    
		\end{equation} 
		For $t>0$, a use of H\"{o}lder's inequality yields
		\begin{align*}
			\|e^{t\Delta}f\|^2_{L^2(B_R(x))}&\leq CR^{\frac{N(p-2)}{p}}\|e^{t\Delta}f(\cdot,t)\|^2_{L^p(B_R(x))}\\
			&\leq CR^{\frac{2\beta+N(p-2)}{p}}t^{-2s}\|f\|^2_{\mathcal{N}^{-2s}_{p,\beta,\infty}(\rn)}
		\end{align*}
		for any  $x\in \rn$ and $R>0$ so that
		\begin{align*}
			\|f\|_{\mathscr{L}^{-1}_{2,N-\lambda}(\rn)}&=\sup_{x\in\rn,R>0}\bigg(\Ball^{\frac{\lambda}{N}-1}\IntC|e^{t\Delta}f(y)|^2dydt\bigg)^{1/2}\\
			&=\sup_{x\in\rn,R>0}\bigg(\Ball^{\frac{\lambda}{N}-1}\int^{R^2}_0 \|e^{t\Delta}f(t)\|^2_{L^2(B_R(x))}dt\bigg)^{\frac{1}{2}}\\
			&\leq C\sup_{x\in\rn,R>0}\bigg(\Ball^{\frac{\lambda}{N}-1}R^{\frac{2\beta+N(p-2)}{p}}\int^{R^2}_0t^{-2s}dt\bigg)^{1/2}\|f\|_{\mathcal{N}^{-2s}_{p,\beta,\infty}(\rn)}\\
			&\leq C\|f\|_{\mathcal{N}^{-2s}_{p,\beta,\infty}(\rn)}.
		\end{align*}
		The proof of the continuous embedding $\mathscr{L}^{-1}_{2,N-\lambda}(\rn)\subset \dot{B}^{-(1+\lambda/2)}_{\infty,\infty}(\rn)$ which holds for any $\lambda\in (-2,2]$ follows from the definition of $\mathscr{L}^{-1}_{2,N-\lambda}(\rn)$ and is inspired by \cite[Proposition 7]{Ca}. Let $\langle\cdot,\cdot\rangle$ denote the duality pairing between $\mathcal{S}(\rn)$ and its dual $\mathcal{S}^{\prime}(\rn)$. There exists a constant $C>0$ such that 
		\begin{align}\label{Y-B1}
			\big|\langle f, e^{-\frac{|x|^2}{4}}\rangle\big|\leq C\|f\|_{\mathscr{L}^{-1}_{2,N-\lambda}(\rn)}
		\end{align}
		for all $f\in \mathscr{L}^{-1}_{2,N-\lambda}(\rn)$. By translation invariance of $\mathscr{L}^{-1}_{2,N-\lambda}(\rn)$, \eqref{Y-B1} implies that 
		\begin{align*}
			\big|(e^{-\frac{|\cdot|^2}{4}}\ast f)\big|\leq C\|f\|_{\mathscr{L}^{-1}_{2,N-\lambda}(\rn)}.    
		\end{align*}
		Moreover, by the invariance of $\mathscr{L}^{-1}_{2,N-\lambda}(\rn)$ with respect to the scaling map $f_{\lambda}(\cdot)=\delta^{\lambda/2+1}f(\delta \cdot)$, $\delta>0$ it holds that
		\begin{align*}
			\sup_{t>0}t^{\frac{\lambda}{4}+\frac{1}{2}}\|e^{t\Delta}f\|_{L^{\infty}(\rn)}\leq C\|f\|_{\mathscr{L}^{-1}_{2,N-\lambda}(\rn)}    
		\end{align*}
		which produces the desired bound.
	\end{proof}
	It is worth pointing out that the membership of a distribution $f$ in $\mathscr{L}^{-1}_{2,N-\lambda}(\rn)$ may be interpreted in terms of Carleson measures. 
	\begin{definition}
		Let $\alpha>0$. A positive measure $\mu$ in $\rn\times \mathbb{R}^+$ is a $\left(parabolic\right)$ $\alpha$-Carleson measure if 
		\begin{equation*}
			\sup_{B\subset \rn}\dfrac{\mu(T(B))}{|B|^{\alpha}}<\infty    
		\end{equation*}
		where the supremum is taken over all balls in $\rn$ and $T(B)$ is the (parabolic) Carleson box $T(B_r(x))=B_r(x)\times (0,r^2]$ for $x\in \rn$ and $r>0$.
	\end{definition}
	By this definition,  it is easy to see that $f$ belongs to $\mathscr{L}^{-1}_{2,N-\lambda}(\rn)$, $N>2$ implies that $d\mu(x,t)=|n(x,t)|^2dxdt$ is a $(1-\frac{\lambda}{N})$-Carleson measure. Thus, $\mathscr{L}^{-1}_{2,N-\lambda}(\rn)$ may be identified with the dual of certain tent space, we refer the interested reader to \cite{Alex}.

\end{document}